\documentclass[11pt,a4paper,reqno]{article}

\usepackage[utf8]{inputenc}

\PassOptionsToPackage{usenames,dvipsnames}{xcolor} 
\usepackage[scale=0.8,centering]{geometry}
\usepackage{amsmath,amssymb,amsthm,amsfonts,pdfpages,graphicx}
\usepackage{color,mathrsfs,tikz}
\usepackage[colorlinks=true,citecolor=RoyalBlue,linkcolor=DarkRed]{hyperref}
\usepackage[font={footnotesize}]{caption}
\usepackage{todonotes}
\usepackage{enumerate}
\usepackage[shortlabels,inline]{enumitem}
\usepackage{caption}
\usepackage{subcaption}
\usepackage{multicol}

\numberwithin{equation}{section}

\usepackage[normalem]{ulem}

\definecolor{Maroon}{cmyk}{0, 0.87, 0.68, 0.32}
\definecolor{DarkRed}{RGB}{165,0,0}  
\definecolor{RoyalBlue}{cmyk}{1, 0.50, 0, 0}
\definecolor{Black}{cmyk}{0, 0, 0, 0}
\definecolor{pinkish}{RGB}{255, 100, 180}
\definecolor{orange}{rgb}{216, 64, 0}

\newcommand{\cC}{{\ensuremath{\mathcal C}} }
\newcommand{\cL}{{\ensuremath{\mathcal L}} }
\newcommand{\cN}{{\ensuremath{\mathcal N}} }
\newcommand{\cW}{{\ensuremath{\mathcal W}} }

\newcommand{\sB}{{\ensuremath{\mathsf B}} }
\newcommand{\sY}{{\ensuremath{\mathsf Y}} }

\newcommand{\N}{\mathbb{N}}
\newcommand{\Z}{\mathbb{Z}}
\newcommand{\R}{\mathbb{R}}

\renewcommand{\tilde}{\widetilde}
\renewcommand{\hat}{\widehat}
\newcommand{\dd}{{\ensuremath{\mathrm d}} }

\newcommand{\I}[1]{\mathbf 1_{\{\bot {#1}\}}}
\newcommand{\1}{\mathbf 1_{\bot}}
\newcommand{\IR}{\mathbf{1}_{\bot,R}}
\newcommand{\hIR}{\hat{\mathbf{1}}_{\bot,R}}

\newcommand{\sign}{\text{sign}}

\renewcommand{\epsilon}{\varepsilon}
\renewcommand{\phi}{\varphi}

\newcommand{\Wa}{\mathcal{W}}  
\newcommand{\Pru}{\mathfrak{P}}  
\newcommand{\dk}{\frac{{\rm d}k}{(2\pi)^d}} 

\newcommand{\sss}   { \scriptscriptstyle }

\newcommand{\e}{\operatorname{e}}

\renewcommand{\phi}{\varphi}

\newcommand{\bk}{{\boldsymbol{\rm k}}}
\newcommand{\bT}{{\boldsymbol{\rm T}}}

\newcommand{\cnT}[1]{{\hat{\boldsymbol{\rm c}}}^{\sss (#1)}}

\newtheorem{theorem}{Theorem}[section]
\newtheorem{definition}{Definition}[section]
\newtheorem{lemma}[theorem]{Lemma}
\newtheorem{corollary}[theorem]{Corollary}
\newtheorem{proposition}[theorem]{Proposition}
\newtheorem{remark}{Remark}[section]

\title{Prudent walk in dimension six and higher}
\author {Markus HEYDENREICH
\thanks{
{Universität Augsburg, Institut für Mathematik, 
86135 Augsburg, Germany.
}
 Email: \texttt{markus.heydenreich@uni-a.de}, Orcid number : 0000-0002-3749-7431
}
\and Lorenzo TAGGI
\thanks{{Sapienza Università di Roma,
Dipartimento di Matematica.
Piazzale Aldo Moro 5,
00186, Roma, Italy.
}
 Email: \texttt{lorenzo.taggi@uniroma1.it}, Orcid number : 0000-0002-7085-9764
 }
\and Niccol\`o TORRI
\thanks{{MODAL'X, UMR 9023, UPL, Univ. Paris Nanterre, F92000 Nanterre France.}
 Email: \texttt{ntorri@parisnanterre.fr}, 
 Orcid number : 0000-0002-4778-1305}
}

\date{}

\begin{document}

\maketitle

\begin{abstract} 
We study the high-dimensional uniform prudent self-avoiding walk, which assigns equal probability to all nearest-neighbor self-avoiding paths of a fixed length that respect the prudent condition, namely, the path cannot take any step in the direction of a previously visited site.
We prove that the prudent self-avoiding walk converges to Brownian motion under diffusive scaling if the dimension is large enough. The same result is true for weakly prudent walk in dimension $d>5$.

A challenging property of the high-dimensional prudent walk is the presence of an infinite-range self-avoidance constraint. 
Interestingly, as a consequence of such a strong self-avoidance constraint, the upper critical dimension of the prudent walk is five, and thus greater than for the classical self-avoiding walk.

\medskip


\smallskip

\end{abstract}


\section{Introduction}
Prudent walk is a class of self-repellent random walks where the walk cannot take increments pointing in the direction of its range. This results in an infinite-range repellence condition.
The prudent walk was originally introduced in \cite{TD87a, TD87b} under the name of \textit{self-directed walk} and in \cite{SSK01} under the name  \textit{outwardly directed self-avoiding walk} as a class of self-avoiding walks which are simple to modelize. 
In the last 20 years this walk has attracted the attention of the combinatorics community, see e.g.\ \cite{BI15,B10,DG08}, and also of the probability community, see e.g.\ \cite{BFV10,CNPT18, PTS16,PT16}.

Let us stress that the prudent condition can be defined in two different ways: in its original formulation \cite{TD87a, TD87b}, prudent random walk is a random walk that chooses its direction uniformly among the admissable moves, and is thus a stochastic process. This model is called \textit{kinetic prudent walk}, and was considered in \cite{BFV10}. 
Alternatively, we fix a length $n$ and choose uniformly a prudent trajectory of length $n$, we call this the \textit{uniform prudent walk}. This latter model has been considered by the combinatorics community and also investigated probabilistically \cite{PTS16, PT16}. In this article we consider the uniform prudent walk.

In dimension $d=2$ the scaling limit of the prudent random walk was identified in \cite{BFV10, PTS16}. 
The present work concentrates on the high-dimensional case. Similar to other walks without self-intersections (most notably self-avoiding walk) there exists an upper critical dimension $d_c$ such that in dimension $d>d_c$ prudent walk is macroscopically very similar to simple random walk. This implies, in particular, that in high dimensions the various self-repellent walk models are all very similar. Most interestingly, we find strong evidence that this upper critical dimension for prudent walk is $d_c=5$, rather than 4 as for self-avoiding walk; the extra dimension results from the infinite-range repellence condition of prudent walks (see Section \ref{sect:criticaldimension} for further discussion of the upper critical dimension). 


Our approach is based on the \textit{lace expansion}. The lace expansion method was introduced by Brydges and Spencer \cite{BrydgSpenc85} to study weakly-self avoiding random walk.
The method was widely exploited for the self-avoiding walk, percolation, lattice trees and lattice animals, and the Ising model, see \cite{HeydeHofst17,  MadraSlade93, Sakai07, SladeBook} and references therein. 
The present work is the first one where we employ the lace expansion to a model with infinite range interaction. 

The lace expansion has also been exploited successfully to investigate critical percolation in high dimension. It might very well be that the methods proposed in this work to allow for an investigation of percolation with infinite-range interactions as studied by Hilario and Sidoravicius \cite{HILARIO20195037}.


\subsection{Main results}
We let $\Wa_n$ be the set of $n$-step nearest-neighbor paths on the hypercubic lattice $\Z^d$ starting at the origin, see \eqref{def:Wn}. We call a path $w=(w(0),w(1),\dots,w(n))$ \emph{prudent} if 
\(w(s)\not\in w(t)+\N_0 \big(w(t)-w(t-1)\big)\)
{for all $0\le s<t\le n$}; see Figure \ref{figPW} for an example. 
\begin{figure}
 \centering
 \includegraphics[scale=1]{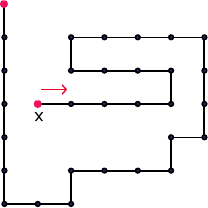}
 \caption{An example of prudent walk path starting at $x$. When traversing the path starting from $x$, the tip of the path at any given moment does not point to the range of the path until that time.}
 \label{figPW}
\end{figure}

Our main result is convergence of the rescaled prudent walk to Brownian motion in high dimension. To this end, let $\mathbb D(A,B)$ be the space of functions $f\colon A\to B$ that are left continuous and have limits from the right, equipped with the Skorokhod $J_1$-topology, see \cite{Billi68}. 

For a prudent walk $w$, we denote the space-time rescaled variable 
\begin{equation}\label{eq:defXn}
X_n(t):=\frac1{\sqrt{Kn}}w\big(\lfloor nt\rfloor\big),\quad t\in[0,1],
\end{equation} 
for a certain constant $K>0$ defined in \eqref{eq:DefK} below.
We consider $X_n$ as a $\mathbb D([0,1],\R)$-valued random variable with respect to the uniform measure $\langle\; \cdot\; \rangle_n$ on the set of $n$-step prudent walks.  

\begin{theorem}[Convergence to Brownian motion]\label{thm:ConvBM}
There exists $d_0>5$ such that for $d>d_0$ the following convergence holds. 
For any bounded continuous function $f\colon \mathbb D([0,1],\R^d)\to\R$, we have that 
\[ \lim_{n\to\infty}\big\langle f(X_n)\big\rangle_n =
\mathbb E\Big[f((W_s)_{s \in [0, 1]})\Big]\, \]
where $(W_s)_{s \geq 0}$ denotes the standard Brownian motion and $\mathbb{E}$ is its expectation.
\end{theorem}

The result remains true for the weakly prudent walk introduced in Section \ref{sec:WeaklyPrudentWalk} in dimension $d>5$ provided that a ``strength parameter'' $\lambda$ for the weakly prudent walk, defined in \eqref{eq:defcnlambda}, is sufficiently small. 

We furthermore prove that the \emph{prudent bubble condition} is satisfied if the dimension is large (or $d>5$ and $\lambda$ is small), see Section \ref{sec:Results}. 

\subsection{Organisation of the paper}
We define prudent walk and weakly prudent walk in Section \ref{sec:ModelResults}. Our results are stated in Section \ref{sec:Results}, followed by a discussion in Section \ref{sec:discussion}. 
In Section \ref{sec:laceexpbound} we derive the lace expansion for prudent walks and obtain diagrammatic estimates on the expansion coefficients. 
Section \ref{sec:bootstrap} establishes the convergence of the expansion and establishes the prudent bubble condition. The analysis presented in these two sections is the major novelty of the present work. Finally,  in Sections \ref{sec:CritG} and \ref{sec:convergence}, we prove the convergence of the prudent walk to a Brownian motion.

\section{Definitions and results}\label{sec:ModelResults}

Given a square-sumable function $f$ on $\mathbb{Z}^d$, we define its Fourier transform as
\begin{equation}
\hat f(k) = \sum\limits_{x \in \mathbb{Z}^d} f(x) e^{i x \cdot k },\qquad k \in [-\pi, \pi]^d
\end{equation}
and observe that 
\begin{equation}
f(x) = \int_{[-\pi, \pi]^d } \frac{\dd^dk}{(2\pi)^d} \hat{f}(k) e^{-i k \cdot x  }.
\end{equation}
We now recall some basic definitions for the simple random walk and introduce the main definitions for prudent walk and weakly prudent walk.

\subsection{Preliminaries on simple random walk}
For $x,y\in\Z^d$ and $n\in\N_0 := \{0, 1, 2, \ldots\}$, we write 
\begin{equation}
\begin{split}\label{def:Wn}
	\Wa_n(x,y):=\Big\{ w\colon\{0,\dots,n\}\to\Z^d\colon & w(0)=x, w(n)=y, \\ &\text{ and } |w(s)-w(s-1)|=1\text{ for all }s=1,\dots,n\Big\}
	\end{split}
\end{equation}
for the set of \emph{$n$-step walks from $x$ to $y$}.  We further write $\Wa(x,y)=\bigcup_{n\in\N_0}\Wa_n(x,y)$, and for $w\in\Wa(x,y)$ we write $|w|$ for the length of the walk, that is, the unique $n\in\N_0$ such that $w\in\Wa_n(x,y)$. 
We denote by $C_z$ the simple random walk Green's function, that is,
\begin{equation}
C_z(x,y)  : = \sum\limits_{n=0}^{\infty}  \big | \mathcal{W}_n(x,y)  \big | z^n,
\end{equation}
and we let $C_z(x)=C_z(0,x)$.
We recall that the radius of convergence of $z\mapsto\sum_{x\in \mathbb Z^d} C_z(x)$ equals $\frac{1}{2d}$.
It is convenient to introduce the function 
\begin{equation}
D(x) : =
\begin{cases}
\frac{1}{2d} & \quad \mbox{ if $|x| = 1$,} \\
0  & \quad \mbox{ otherwise},
\end{cases}
\qquad x\in\Z^d,
\end{equation}
with Fourier transform given by 
$\hat{D}(k) = \frac{1}{d} \sum_{i=1}^{d}  \cos(k_i)$. 
Since $|\Wa_n(0,x)|=(2d)^n D^{\ast n}(x)$ (where $D^{\ast n}=D\ast\cdots\ast D$ denotes $n$-fold convolution), we get that the Fourier transform of $C_z(x)$ for  $|z|<z_c$ is given by 
\begin{equation}\label{eq:connectionCD}
\hat{C}_{z}(k) = \frac{1}{1 -  2 d z \hat{D}(k)},\qquad k\in [-\pi, \pi]^d. 
\end{equation}
For later reference, we note the elementary relations 
\begin{equation}\label{eq:Dk}
	1-\hat D(k)=\Big( \frac{1}{2d} + o(1)\Big)\,|k|^2
	\quad \text{as $|k|\to 0$}\qquad {and} \qquad
	\lim_{n\to\infty}n \big (1-\hat D(k/\sqrt n)  \big ) =  \frac{1}{2d} \,|k|^2,
\end{equation}
where $| \, \cdot \, |$ denotes the $L^2$ norm.

\subsection{The prudent walk}
We now properly define the prudent walk. 
\begin{definition}
A walk $w\in\Wa_n(x,y)$ satisfies the \emph{prudent condition} if 
\begin{equation}\label{eq:prudent}
	w(s)\not\in w(t)+\N_0 \big(w(t)-w(t-1)\big)
	\qquad\text{for all $0\le s<t\le n$.} 
\end{equation}
\end{definition}
We shortly say ``$w$ is prudent''. 
The set of \emph{prudent} $n$-step walks from $x$ to $y$ is denoted by $\Pru_n(x,y)$ and 
$c_n(x,y):=|\Pru_n(x,y)|$ denotes its cardinality. We denote $c_n(x):=c_n(0,x)$.
The total number of $n$-step prudent walks is $c_n:=\sum_{x\in\Z^d}c_n(x)$. From sub-additivity it follows that  the limit
\begin{equation}
\mu:=\lim_{n\to +\infty}(c_n)^{1/n}
\end{equation}
exists. The constant $\mu$ is called the \textit{connective constant} and satisfies the trivial bounds
\begin{equation}\label{eq:inequalityconnective}
d - 1 \leq \mu \leq  2d-1.
\end{equation}

Prudent random walk is the uniform measure on the set of $n$-step prudent walks. Note that the prudent random walk is not a stochastic process, because the resulting family of measures is not consistent. 

The \emph{prudent two-point function} is defined as the generating function
\begin{equation}\label{eq:G}
	G_z(x,y)=\sum_{n=0}^\infty c_n(x,y)\,z^n
\end{equation}
and we let $G_z(x)=G_z(0,x)$.
We let  $z_c : =\frac{1}{\mu}$ be the radius of convergence of the power series
\begin{equation}\label{eq:susceptibility}
\chi(z) := \sum_{n=0}^\infty c_n\,z^n = \sum\limits_{x \in \mathbb{Z}^d} G_z(x).
\end{equation}
We refer to $z_c$ as \textit{critical point} and $\chi(z)$ as \textit{susceptibility}.

\subsection{The weakly-prudent walk} \label{sec:WeaklyPrudentWalk}
We now introduce the weakly-prudent walk, in which each step of the walk which does not fulfill the prudent condition is penalized by a multiplicative parameter $\lambda \in [0,1]$. Hence the case $\lambda = 0$ corresponds to simple random walk, while the case $\lambda = 1$ corresponds to  prudent walk.
We also  rephrase the expression for the prudent two-point function in \eqref{eq:G} so that it is more suited for an expansion. 
Given a $n$-step walk $w=(w(0),\dots,w(n))$, and $0\le s<t\le n$, we define
\begin{equation}\label{eq:Ust}
	U_{st}(w) : =\begin{cases}
	-1,& \quad \text{if}\quad w :  w(t)\downarrow w(s),\\
	0,&\quad \text{otherwise},
	\end{cases}
\end{equation}
where $w(t)\downarrow a$ is shorthand for $a-w(t)\in \N_0 \big(w(t)-w(t-1)\big)$---recall the prudent condition \eqref{eq:prudent}.
If $w$ is such that  $w(t)\downarrow a$, then we say that 
the $t$-th step of $w$ \emph{sees} $a$ or, shortly, $w(t)$ \emph{sees} $a$.
The two-point function for the  weakly-prudent walk is defined as 
\begin{equation}\label{eq:GU}
	G^\lambda_z(x,y)=\sum_{n \geq 0 }\ z^n c^\lambda_n(x,y) 
\end{equation}
where 
\begin{equation}\label{eq:defcnlambda}
c^\lambda_n(x,y) = \sum_{w\in\Wa_n(x,y)}
\phi^\lambda(w) \quad \quad \text{and}\quad 
\phi^\lambda(w) := \prod_{0\le s<t\le n}  \big (1+ \, \lambda \, U_{st}(w) \big ).
\end{equation} 
We also let $c^\lambda_n(x)=c^\lambda_n(0,x)$ and $G^\lambda_z(x)=G^\lambda_z(0,x)$.

We let $c^\lambda_n=\sum_{x\in \mathbb Z^d}c_n^\lambda(x)$. Let us observe that $\phi^1(w)\neq 0$ if and only if $w$ satisfies the prudent condition. To simplify the notation, we denote $\phi(w)=\phi^1(w)$.

In analogy with \eqref{eq:susceptibility}, for any $\lambda>0$, we define $z_c(\lambda)$ as the radius of convergence of the susceptibility
\begin{equation}
\chi^{\lambda}(z) : = \sum\limits_{x \in \mathbb{Z}^d} G_z^{\lambda}(x) = \sum_{n \geq 0} \ z^n\sum_{w\in\Wa_n}\phi^\lambda(w).
\end{equation}
Hence it follows that $z_c(0) = \frac{1}{2d}$, $z_c(1) = \frac{1}{\mu}$, $G^1_z(x) = G_z(x),$
and
$\chi^1(z) = \chi(z)$.

\subsection{Results} \label{sec:Results}

A central object in our analysis is the \textit{prudent bubble diagram}, which we now introduce. 
To this end, we introduce a notion that is weaker than the symbol $w(t)\downarrow a$ used in \eqref{eq:Ust}:
for $x,a\in\Z^d$, $a\ne x$, we let 
\begin{equation}\label{def:abotx}
x\bot a  \quad \Longleftrightarrow \quad x-a\in\bigcup_{j=1,\dots,d}e_j\,\Z \, ,
\end{equation}
where $e_1,\dots,e_d$ are the coordinate axes (i.e., $x\bot a$ whenever $x$ and $a$ differ in at most one coordinate). 
For any $x,a\in\Z^d$ we define the \textit{modified indicator function}
\begin{equation}\label{ourindicator}
	\I{a}(x) : = 
	\begin{cases}
	    \frac1d & \mbox{ if $x \bot  a$ and $x \neq a$} \\
		0 & \text{ otherwise}, 
	\end{cases}
\end{equation}
and abbreviate $\1(x)=\I{0}(x)$. Observe that $\I{a}(x)=\1(x-a)$.
\begin{definition}[Prudent Bubble Diagram]
\label{def:PrudentBubble}
We define the prudent bubble diagram as
\begin{equation}\label{def:bubbleB}
\sB^\lambda_z   :   = 
 \|   G^\lambda_z \ast  G_z^\lambda  \ast \1  \|_{\infty}  = 
   \sup_{y\in \mathbb{Z}^d}
\sum\limits_{ \substack{ x_1, x_2 \in \mathbb{Z}^d }}  
G^\lambda_z(0,x_1) 
G_z^\lambda(x_1, x_2) 
\I{ y } (x_2)  
\end{equation}
If $\lambda=1$, we write $\sB_z =\sB^1_z $.
\end{definition}
The next theorem states that the critical prudent bubble diagram is finite if the dimension is sufficiently large and $\lambda=1$ or if the dimension is at least $6$ and $\lambda$ is sufficiently small. 
\begin{theorem}[Prudent Bubble Condition]\label{thm:BubbleCond}
There exist $d_0\ge5$ and $\lambda_0>0$ such that if either 
\begin{enumerate}
\item[{\rm (a)}] $d>5$ and $\lambda\in (0,\lambda_0)$, or
\item[{\rm (b)}] $\lambda=1$ and $d>d_0$,  
\end{enumerate} 
then the critical prudent bubble diagram $\sB^\lambda_{z_c}$ is finite. Moreover, there exists a constant $C>0$ (independent of $\lambda$ and the dimension $d$) such that 
	\begin{equation}
	\sB^\lambda_{z_c}<C/d.
	\end{equation}
\end{theorem} 
The proof of Theorem \ref{thm:BubbleCond} is intertwined with the asymptotic behaviour of the critical two-point function, which we formulate in the next theorem. 

\begin{theorem}[Critical Two-Point Function]\label{thm:CritG}
Under the conditions of Theorem \ref{thm:BubbleCond}, there exists $K \in (0, \infty)$, 
defined in (\ref{eq:DefK}) below, such that for $k\in\R^d$,
\begin{equation}\label{eq:fcnConv}
	\hat c^\lambda_n\big(k/\sqrt{n}\big)\sim\hat c^\lambda_n\big(0\big)\exp\{-K\,|k|^2\} \qquad\text{ as $n\to\infty$.}
\end{equation}
\end{theorem}
Theorem \ref{thm:CritG} is already a strong indication towards the Brownian motion limit as it establishes a Gaussian limit for the endpoint. We reformulate this convergence in the generalized context of Section \ref{sec:WeaklyPrudentWalk}. 

\begin{theorem}[Convergence to Brownian motion; general version]\label{thm:ConvBMgen}
Consider (weakly or strictly) prudent walk in dimension $d>5$. 
There exists $c>0$ such that if $\lambda/d<c$, then
\[ \lim_{n\to\infty}\big\langle f(X_n)\big\rangle_n =
\mathbb E\Big[f((W_s)_{s \in [0,1]})\Big]\, \]
for any bounded continuous function $f\colon \mathbb D([0,1],\R^d)\to\R$, where 
$(W_s)_{s \geq 0}$ denotes the standard Brownian motion and $\mathbb{E}$ is its expectation.
\end{theorem}
This theorem generalizes Theorem \ref{thm:ConvBM}. 
In order to get convergence in path space, as claimed in Theorem \ref{thm:ConvBMgen}, we need both convergence of finite-dimensional projections as well as tightness. We state and prove these results in Section \ref{sec:convergence}. 

Further results concern asymptotics of the critical Green's function of prudent walk in Corollary \ref{lem:GzNull} and asymptotics of the moment of the endpoint of prudent walk in Proposition \ref{prop-moments}.

\subsection{Discussion}\label{sec:discussion}

It is instructive to compare our results for prudent walk with analog results for (classical) self-avoiding walk. Indeed, it is known that rescaled self-avoiding walk in sufficiently high dimension converges to Brownian motion. This has been established by Slade \cite{Slade88,Slade89} using memory cutoff, by Hara and Slade \cite{HaraSlade92} using an alternative method involving fractional derivative estimates, by Heydenreich \cite{Heyde11} for long-range self-avoiding walk and, very recently, by Michta \cite{Micht22} for weakly self-avoiding walk on the high-dimensional torus. 

An interesting feature of our results is that the upper critical dimension, which appears to be $5$ for prudent walk, differs from the upper critical dimension of self-avoiding walk, which is $4$. 
We give a heuristic explanation for this change in the next subsection. 


Our proof of the bubble condition and the critical two-point function requires a novel lace expansion argument. Our proofs are based on the lace expansion, which was pioneered by Brydges and Spencer \cite{BrydgSpenc85} and further developed for self-avoiding walk by Hara and Slade. While the actual expansion is the same as for self-avoiding walks, the expression for the $\Pi$-diagrams as defined in \eqref{def:laceexpansion} involves both the two-point function $G_z^\lambda$ as well as the (modified) indicator function $\1$  (see Figure \ref{figN}, where indicators are represented by dotted lines), and interestingly the resulting diagrams are more reminiscent of percolation lace-expansion diagrams rather than the the diagrams appearing for classical self-avoiding walk. 
However, in contrast to earlier lace expansions, we obtain inhomogeneous lace expansion diagrams, and dealing with these inhomogeneities as well as the the infinite-range avoidance constraints is the major challenge in the proof. Our analysis therefore has substantial differences with respect to such classical cases.

This heterogeneity is reflected in the definition of the prudent bubble diagram, and dealing with it is the main obstacle in deriving the diagrammatic bounds and in the proof that the expansion converges. We demonstrate these in Sections \ref{sec:laceexpbound} and \ref{sec:bootstrap}, which we consider the most innovative part of the paper. Once the above bounds have been established, we can rely on a well-established machinery to get convergence to Brownian motion following \cite[Chapter 6.6]{MadraSlade93} or \cite{Heyde11}. Therefore, we concentrate in Section \ref{sec:convergence} and Appendix \ref{sec:moments} on the necessary adaptations rather than detailing the argument.


Recently, novel proofs have been found that allow a direct comparison between weakly self-avoiding walk's Green function and simple random walk's Green function, cf.\ \cite{BHK15,Slade22}. 
It might be asked whether similar techniques are applicable to weakly prudent walk. However, it appears that the $\Pi$-diagram for (weakly) prudent walk is converging as $|\Pi(x)|\approx|x|^{-(d-2)}$ along the coordinate axes, and thus decays too slow for the methods of \cite{BHK15,Slade22} to apply
 straightforwardly. 
Even though there is presumably faster decay off the axes, this is therefore indicating that methods involving pointwise estimates are not naturally adapted to the prudent walk and it should require additional estimates to fit the prudent walk setting. Our analysis of the lace expansion is based on Fourier transforms.

\subsection{The upper critical dimension for the prudent walk}
\label{sect:criticaldimension}
Our results show that the weakly prudent walk converges to Brownian motion under diffusive scaling  in dimension $d > 5$ if the interaction parameter $\lambda$ is small enough. 
Since the critical dimension is not expected to depend
on the intensity of the repulsion parameter,  this strongly suggests
that also the (strictly) prudent walk converges to Brownian motion under diffusive scaling
in any dimension $d >5$ (which we prove only if the dimension is large enough),
and, thus,  that the critical dimension is at most five. 

We conjecture that the upper critical dimension for the prudent walk is precisely five.
This implies that it is   \textit{strictly greater} than for the self-avoiding walk, which has  upper critical dimension four. 
Our  conjecture is supported by the following considerations. 

It is generally expected that the prudent bubble diagram in Definition \ref{def:PrudentBubble} is finite whenever the corresponding prudent \emph{random walk} bubble diagram $\|C_1\ast C_1\ast \1\|_\infty$ is finite. Indeed, the latter is finite if and only if $d>5$, because $(C\ast C)(x)\approx |x|^{4-d}$ and this is summable over $x\in e_1\Z$ whenever $4-d<-1$. This suggests that also the (ordinary) prudent bubble diagram is infinite for $d\le5$. 

Alternatively, one may study the expected number of times that two simple random walks ``see each other'': 
Let  $(X^1(t))_{t \in \mathbb{N}_0}$, $(X^2(t))_{t \in \mathbb{N}_0}$  be two independent simple random walks starting from the origin, and 
let $X^i_j$ denote the $j$th coordinate of $X^i$.  By the Markov property and reversibility,
\begin{align*}
& \mathbb E \Big (  \sum\limits_{s, t >0}^{\infty}   \boldsymbol{1}_{  \{ X^1(t) \downarrow  X^2(s) \}   } \Big )  
 = 
 \frac{1}{2d}   \mathbb E \Big (  \sum\limits_{s > 0, t > 0}^{\infty}   \I{X^2(s)}\big (X^1(t-1) \big )  \Big )  
   =   \frac{1}{2d}  \mathbb E \Big (  \sum\limits_{s > 0, t \geq 0}^{\infty}   \I{0}\big (X^1(t + s) \big )  \Big ) \\
&\qquad
 =   \frac{1}{2d}  \mathbb E \Big (  \sum\limits_{t>0}^{\infty} \, t \,   \I{0}\big (X^1(t) \big )  \Big ) 
 =    \frac{1}{2}  \mathbb E \Big (  \sum\limits_{t > 0}^{\infty} \, t \,   
\boldsymbol{1}_{  \big \{ X_j^1(t)   = 0 \forall j \neq d,  \, \, X_d^1(t) \neq 0 \big  \}}   \Big )  \sim  C  \sum\limits_{t=1}^{\infty}  t^{ - \frac{d-1}{2}  + 1 },
\end{align*} 
which is finite if and only if $d > 5$. 

The above calculation suggests that models with $\ell$-dimensional avoidance constraints have upper critical dimension $4+\ell$. 
Self-avoiding walk may be viewed as a repellent model with a 0-dimensional avoidance constraint (and has upper critical dimension 4), whereas prudent walk has a one-dimensional avoidance constraint. 

We observe a similar shift in the upper critical dimension for oriented percolation (which has upper critical dimension 4+1) but that is caused by a change of the infrared bound (e.g.\ \cite[Thm.~12.1]{SladeBook}), whereas in our model the shift is caused by a change in the bubble diagram.

\section{Lace Expansion and diagrammatic estimates}\label{sec:laceexpbound}
We next derive the lace expansion for prudent walks, and derive \emph{diagrammatic bounds} that give quantitative control over the lace expansion-related quantities. 
As we noted above, the presence of non-localized infinite-range avoidance constraints results in a new type of inhomogeneous lace-expansion diagrams, and therefore the lace-expansion analysis varies significantly from classical self-avoiding walk. 
%

\subsection{Lace Expansion}
\label{sec:laceexpansion}
We shall next recall the lace expansion for self-intersecting walks. We follow the presentation in Slade's exposition about self-avoiding walks \cite[Chapter 3]{SladeBook} or \cite[Chapter 5]{MadraSlade93}).
\begin{definition}[Graphs and Laces] ~
\label{def:graphslaces}
\begin{itemize}
\item Given an interval $I =[a,b]$ of integers ($a,b\in\Z$, $0\le a<b$), we refer to a pair
$\{s,t\}$, $s < t$ of elements of $I$ as an edge. To abbreviate the notation, we usually write $st$ for $\{s,t\}$. 
The length of an edge $st$ is $t-s$. A set of edges is called a \emph{graph}. 
\item A graph $\Gamma$ is said to be \emph{connected} if both $a$ and $b$ are end-points of edges in $\Gamma$, and if in addition, for any $c \in (a,b)$, there are $s, t \in [a,b]$ such that $s < c < t$ and $st \in \Gamma$. The set of all graphs on $[a,b]$
is denoted by $\mathcal{B}[a,b]$, and the subset consisting of all connected graphs is denoted by $\mathcal{G}[a,b]$. 
\item A \emph{lace} is a minimally connected graph, i.e, a connected graph for which the removal of any edge would result in a disconnected graph.
The set of laces on $[a,b]$ is denoted by $\mathcal{L}[a,b]$
and the set of laces on $[a,b]$ consisting of exactly $N$ edges
is denoted by $\mathcal{L}_N[a,b]$.
\end{itemize}
\end{definition}
See Figure \ref{fig2} for an illustration of the terms \emph{graph} and \emph{lace}, and \cite[Chapter 3]{SladeBook} for further explanation about these quantities. 

Given a connected graph $\Gamma$, the following prescription associates to $\Gamma$ a unique lace $\mathcal{L}_{\Gamma}$: the lace $\mathcal{L}_{\Gamma}$ consists of edges 
$s_1 t_1$, $s_2 t_2$, $\ldots$ where 
\begin{align}\label{eq:constructionst}
\begin{aligned}
s_1 & = a, \\  t_1  & = \max \{  t : at \in \Gamma \} 
\end{aligned}
\qquad \text{and for $i\ge 2$}\qquad
\begin{aligned}
  t_{i+1}  &=  \max \{  t : st \in \Gamma, s < t_i \} \\
s_i &  =  \min \{  s : st_i \in \Gamma \}.
\end{aligned}
\end{align}
We let $\mathcal{C}(L)$ be the set of all edges which are `compatible' with the lace $L$, that is, all edges $st\not\in L$ such that $\cL_{L\cup\{st\}}=L$.
We finally introduce the \textit{lace expansion coefficients} as 
\begin{align}\label{def:laceexpansion}
\Pi^\lambda_z(x)=&\sum_{N=1}^\infty (-1)^N \Pi_z^{(N)}(x),
\end{align}
where
\begin{equation}\label{eq:PiN}
\Pi_z^{(N)}(x) =  
	\sum_{w\in \Wa(0,x)} 
	z^{|w|} \sum_{L\in\cL_N[0,|w|]} \, 
\Big (  \prod_{st\in L}  -\lambda  U_{st}(w)  \Big ) \,
\Big (  \prod_{s't'\in \cC(L)} \big (1+ \lambda U_{s't'}(w) \big ) \, \Big ),
 \end{equation}
(the sum is over all -not necessarily prudent- walks). 
While $\Pi_z^{(N)}(x)$ depends on $\lambda$, we suppress this in the notation. 

\begin{proposition}[Lace expansion]\label{theo:recursionforG}
Let $\lambda \in [0,1]$, 
and $|z|<z_c(\lambda)$, then
\begin{equation}\label{eq:decompositionGx}
G^\lambda_z(0,x) = \delta_{0,x} + z \sum\limits_{u \sim 0} G_z^\lambda(u,x) + \sum\limits_{v \in \mathbb{Z}^d} \Pi^\lambda_z(0,v) \, G_z^\lambda(v,x),
\end{equation}
where $u\sim 0$ means that $u$ and $0$ are neighbors (i.e., $|u|=1$). 
\end{proposition}
This result is well-known in the lace-expansion literature for self-avoiding walks, see  \cite[Section 5.4]{MadraSlade93} or \cite[Chapter 3]{SladeBook}. Let us stress that the statement of the theorem is independent of the precise interpretation of $U_{st}$, and concerns only the underlying graph structure that we use to decompose the two point function $G_z$. The proof is thus analogous to that of \cite[Theorem 5.2.3]{MadraSlade93}.

\begin{figure}
 \centering
 \includegraphics[scale=0.20]{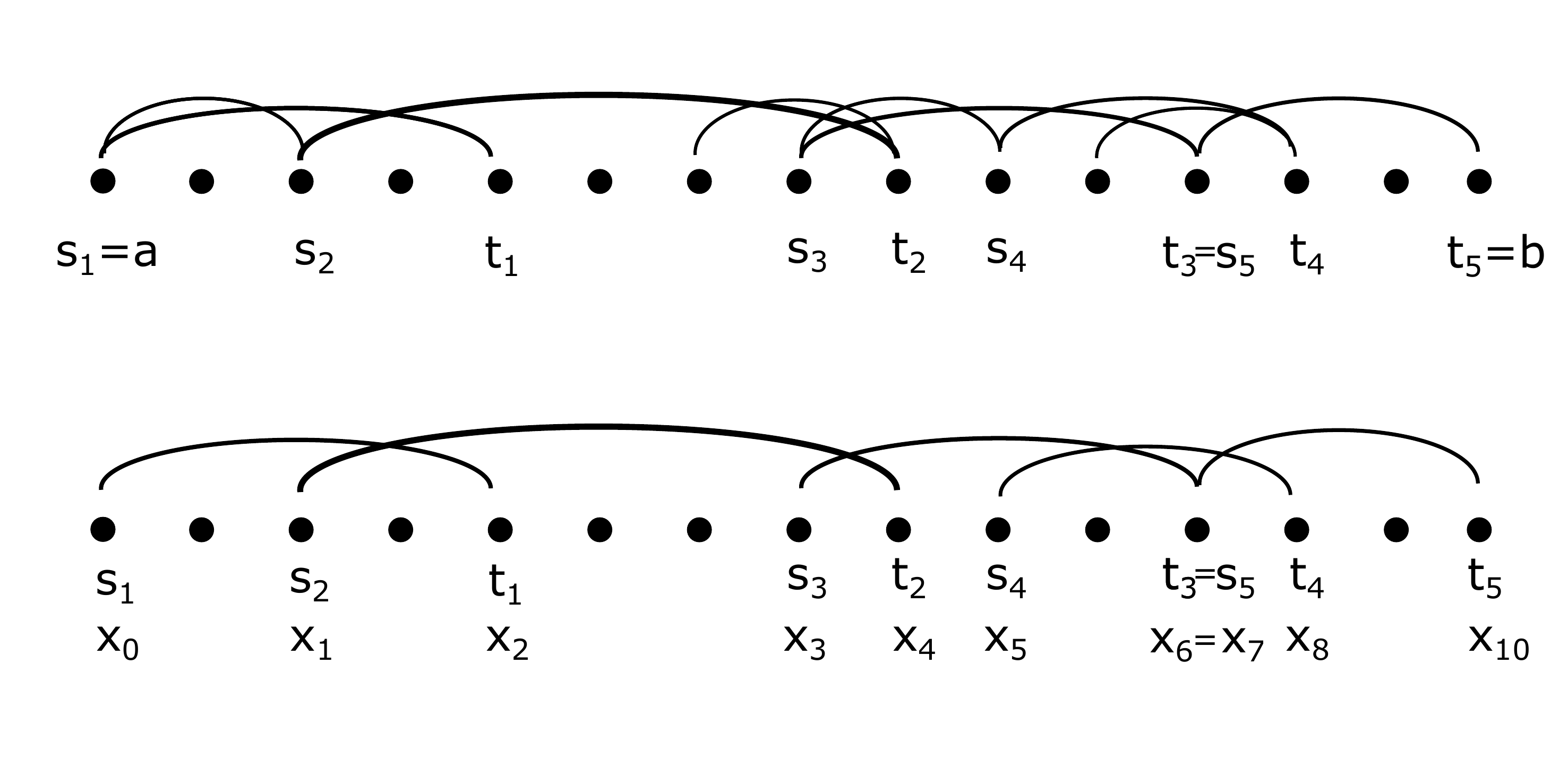}
 \caption{\textit{Top:} an example of graph $\Gamma$ in the interval $I = [a,b]$. 
 \textit{Bottom:} the corresponding lace $\mathcal{L}_\Gamma$.
 }
 \label{fig2}
\end{figure}

\subsection{Diagrammatic bounds} 
\label{sec:diagrammatic}
In the whole section we fix $\lambda \in \big (0, 1 \big ]$ and  $z\in[0,z_c(\lambda))$. 
We recall the bubble diagram, defined in \eqref{def:bubbleB}, 
and define for all $x \in \mathbb{Z}^d$ and $k \in [-\pi, \pi]^d$ the function
\begin{equation}\label{eq:defHzk}
\sY^\lambda_{z,k}(x):= [1 - \cos(k \cdot x)] \, G^\lambda_z(x). 
\end{equation}
We then have the following result which estimates $\Pi_z^{(N)}$ (and thus $\Pi^\lambda_z$)
 in terms of $\sB^\lambda_z$ and $\sY^\lambda_z$.
\begin{proposition}[Diagrammatic Bounds]\label{thm:equi41}
 For any $N \geq 1$, 
\begin{align}\label{eq:bounds2}
\sum\limits_{ x \in \mathbb{Z}^d}
  \Pi_z^{(N)}(x)  & \leq   (dz\lambda)^N   \,   (\sB_z^\lambda )^{N},
  \\
\label{eq:bounds2b}
\sum\limits_{ x \in \mathbb{Z}^d }
[ 1 - \cos(k \cdot x)    ] \, \Pi_z^{(N)}(x)  & \leq \,  
(dz\lambda)^N 
(3 N-1) (2 N - 1) 
\| \sY^\lambda_{z,k}* \1   \|_{\infty} (\sB^{\lambda}_z)^{N-1} \\ \nonumber
&  \quad \quad \quad + (d z\lambda )^N (3 N - 1) N (\sB^{\lambda}_z)^N (1 - \hat{D}(k)).
\end{align}
\end{proposition}
The rest of the section is dedicated to the proof of Proposition \ref{thm:equi41}, which is based on the representation in \eqref{eq:PiN}. 
We first prove the statement 
for $N=1$, and subsequently for $N\ge 2$.
In order to lighten the notation, we suppress the dependence on  $\lambda$ and $z$ from the notation. 

\medskip

\proof[Proof of \eqref{eq:bounds2} and \eqref{eq:bounds2b}, {\bf case $N=1$}.]
We start with \eqref{eq:bounds2}. We use \eqref{eq:PiN} for $N=1$, in which case each lace consists of only a single edge from $0$ to the last step of the walk $|w|>0$.
This implies that the last step of each walk $w$ in \eqref{eq:PiN}  is such that  $w( |w| ) \downarrow 0$, namely 
 $w( |w| )- w( |w| - 1 )  \in e_i \mathbb{Z}$ for some $i \in [d]$.
In turn, this implies that the second-last step of the walk lies on one of the Cartesian axes,
$w( |w| - 1 ) \bot 0,$
and differs from the origin.
If $x = (0, \ldots, x_i, \ldots,  0 )$ and $x_i \neq 0$, then, from the definition of $\Pi^{(1)}(x)$ in \eqref{eq:PiN} we get
\begin{align}\label{eq:Pi1x}
	\Pi^{(1)}(x)
    \le d z  \lambda  \,G \big (0, x + e_i \, \sign(x_i) \big ) \1{ (x + e_i) }.
    \end{align}
Moreover,
    \begin{align}
	\Pi^{(1)}(0)
   \le z\lambda \,   \sum_{y \sim 0} \,G \big (0, y \big ) 
= d z  \lambda \,   \sum\limits_{y \sim 0} \,G\big (0, y \big ) \1(y),
    \end{align}
    Finally, if $x$ does not belong to the Cartesian axes, then $\Pi^{(1)}(x) = 0$. By summing over $x$ in $\mathbb{Z}^d$, using the inequalities above and finally using $\delta_{0,x}\le G(x)$, we deduce
    $$
    \sum\limits_{x \in \mathbb{Z}^d} \Pi^{(1)}_z(x) 
    \leq d z \lambda \, (G * \1) (0) 
    \leq d z \lambda \, (G * G * \1) (0) 
    \leq d z \lambda  \,  \sB_z^\lambda,
    $$
    thus obtaining (\ref{eq:bounds2}) when $N=1$.
    
    \smallskip
    
We now prove \eqref{eq:bounds2b}. 
First we use the fact that $\Pi_z^{(1)}(x) = 0$ if $x$ does not fulfill $x \bot 0$ and thus
\begin{equation}\label{eqPi1Bd1}
\sum\limits_{x \in \mathbb{Z}^d}
[ 1 -  \cos(k \cdot x)    ] \, \Pi_z^{(1)}(x)  
=
\lambda z  \sum\limits_{i=1}^{d} \sum\limits_{x \in e_i \mathbb{Z} \setminus \{0\}} G(0,x ) \big ( 1 - \cos ( k \cdot (x- \sign(x_i) e_i)   ) \big  ). 
\end{equation}
Here, $x$ corresponds to the second-last point of the walk and $z$, corresponding to the weight of the last step of the walk, has been factorised as in  \eqref{eq:Pi1x}.

A key ingredient in bounding  \eqref{eq:bounds2b} is given by the following formula, proven as Lemma 2.13 in \cite{FitznHofst17}: 
for $n \in \mathbb{N}$, $x_1, \ldots, x_n \in \mathbb{Z}^d$, and $k\in\R^d$, 
    \begin{multline}\label{eq:inequalitycos}
    1 - \cos(k \cdot x_n) \leq n \Big [ 
    (1 - \cos (k \cdot x_1))  +     (1 - \cos (k \cdot (x_2-x_1)) +
    \\ 
    \cdots + (1 - \cos (k \cdot (x_n-x_{n-1}))
    \Big ].
    \end{multline}
We first use (\ref{eq:inequalitycos}) for $n=2$ with  $x_1=x$ and $x_2=x \pm e_i$, obtaining that
    \begin{equation*}
    1 - \cos \big  (k \cdot (x \pm  e_i) \big ) \leq 2 \big [ \big  ( 1 - \cos (k \cdot x    \big  ) 
    + \big(1 - \cos (k_i) \big)\big ].
    \end{equation*}
Inserting this bound into \eqref{eqPi1Bd1}, we obtain that 
\begin{equation*}
\begin{split}
\sum\limits_{x \in \mathbb{Z}^d}
[ 1 - & \cos(k \cdot x)    ] \, \Pi_z^{(1)}(x)  
\\
& \leq 2 z \lambda 
\sum\limits_{i=1}^{d} \sum\limits_{x \in e_i \mathbb{Z} \setminus \{0\}} G(0,x) ( 1 - \cos ( k \cdot x   ) )
+ 2 z  \lambda 
\sum\limits_{i=1}^{d} 
\sum\limits_{x \in e_i \mathbb{Z} \setminus \{0\}} G(0,x ) ( 1 - \cos ( k_i ) )  \\
& \leq 2 d z \lambda   \, 
 \| \sY_{k,z}^\lambda \,  * \, \1   \|_{\infty} 
 \, + \, 2 d z  \lambda  \|  G_z^\lambda *  \1    \|_{\infty} (1-\hat D(k))\, ,
 \end{split}
\end{equation*}
where in the last inequality we use the symmetry to get that $\sum_{x \in e_i \mathbb{Z} \setminus \{0\}} G(0,x ) =  \sum_{x  \in \mathbb Z^d} G(0, x ) \1(x)$.
Since $ \|  G_z^\lambda *  \1    \|_{\infty} \leq \sB^\lambda_z$, this concludes the proof of the case $N=1$.
\qed

\medskip

{\bf Case $N\ge 2$.} In order to prove \eqref{eq:bounds2} for $N\ge2$, we first formulate an auxiliary lemma. 
To this end, we introduce the function $\mathcal N_{a,b}(x)$ for $a, b, x \in \mathbb{Z}^d$, as
\begin{equation}\label{eq:defN}
\mathcal{N}_{a,b} (x) := 
\begin{cases}
\frac{1}{d}  & \mbox{ if $a \bot b$,   $a \neq b$,  $x\in \overline{ab}$, $x \sim a$ ,}\\
0 & \mbox{ otherwise,} 
\end{cases}
\end{equation}
where $\overline{ab}$ is the segment in $\mathbb{R}^d$ with end-points  $a$ and $b$, while $a\bot b$ is the condition introduced in 
\eqref{def:abotx}. We will often use the following identity,
\begin{equation}\label{eq:propertyN}
\sum\limits_{ x \in \mathbb{Z}^d}  \mathcal{N}_{a,b} (x)  =  \I{b}(a) =  \1(a-b),
\end{equation}
where $ \I{b}(a)$ is defined in \eqref{ourindicator}. We use the function $\mathcal{N}_{a,b} (x) $ to identify the points that violate the prudent condition in order to use the lace expansion.
We then have the following upper bound for $\Pi_z^{(N)}(x)$:
\begin{lemma}\label{lemma:lemmaN}
 For any $N \geq 2$ we let ${\bf x}^{N}= (x_1, x_2^\prime, x_2, x_3, x_4^{\prime}, x_4,
\ldots, x_{2N-2},  x^\prime_{2N-1}, x_{2N-1})\in (\mathbb{Z}^d)^{3N-1}$ with $x_{2N-1}=x$. 
Then
 \begin{equation}\label{eq:lemmaN}
 \begin{split}
 \sum\limits_{x \in \mathbb{Z}^d}  \Pi_z^{(N)}(x)  
& 
\leq (z d \lambda)^N  \sum\limits_{  {\bf x}^{N}   \in (\mathbb{Z}^d)^{3N-1} }  \Bigg\{ 
 \prod_{i\in \{1, \ldots, N-1\} }
\Big(
G(x_{2 i-2},x_{2i-1}) 
 G(x_{ 2 i -1 },x^\prime_{2 i})
\mathcal{N}_{x^\prime_{2 i}, x_{2i-3}}(x_{2i})
\Big)
 \\ & \hspace{4cm} \times G(x_{2N - 2}, x_{2N-1}^\prime) \, \,
 \mathcal{N}_{x_{2N-1}^\prime,x_{2N-3}}(x_{2N-1})
\Bigg\},
\end{split}
\end{equation}
where we use the convention that $x_{-1}=x_0 = 0$.
\end{lemma}
We refer to Figure \ref{figN} for a diagrammatic representation of \eqref{eq:lemmaN}. 

 \begin{figure}
\centering
 \includegraphics[scale=1.5]{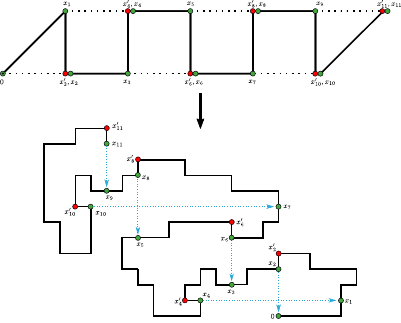}
 \caption{Diagrammatic bounds in case $N = 6$. Solid lines represent $G$, whereas dotted lines represent $\1$. At the bottom a (strictly) prudent path associated with the diagrammatic bounds is represented. Blue arrows mark the dotted lines of the diagram.}
 \label{figN}
\end{figure}

\begin{proof}[Proof of Lemma \ref{lemma:lemmaN}]
We use \eqref{eq:PiN} and decompose the prudent walk $w$ into $2N-1$ prudent sub-walks,
$w^1\in\Wa(0,x_1), w^2\in\Wa(x_1,x_2),\dots,w^{2N-1}\in\Wa(x_{2N-2},x_{2N-1})$.
We first sum over the end-points of such sub-walks, 
 $x_1, x_2, \ldots x_{2N-1} = x$,
and then sum over the sub-walks connecting such points.
 By \eqref{eq:Ust} and \eqref{eq:PiN},  the decomposition is such that 
the walk $w^2$ sees the origin at its last step (recall definition of \emph{seeing} after \eqref{eq:Ust}).
Moreover, for each even integer $i \in [1, 2N-2]$ 
the prudent walk $w^{i}$ sees the point 
$x_{i-3}$ at its last step, 
and the last walk 
$w^{2N-1}$ sees $x_{2N-3}$ at its last step.
For each even integer 
 $i \in [2, 2N-2]$ we denote by $x_i^\prime$ the location of the second-last step of the walk $w^i$, and we denote
by $x^\prime_{2N-1}$ the location of the second-last step of the walk
$w^{2N-1}$ (which is then a neighbour of $x_{2N-1}=x$).
Let us observe that
 the walk $w^2$ sees $0$ at its last step, thus
 $x_2^\prime \bot 0$, 
$x_2^\prime \neq 0$,
and $x_2 \bot 0$.
Similarly, since for each even integer $i \in [4, 2N-2]$ as well as for $i=2N-1$ the walk $w^i$ sees $x_{i-3}$ at its last step, thus $x^\prime_i \bot x_{i-3}$,
$x^\prime_i  \neq x_{i-3}$
 and $x_i \bot x_{i-3}$. Moreover, the last step of the sub-walks  $w^i$ for each even $i \in [2, 2N-2]$ and for the sub-walk $w_{2N-1}$ starts at $x^\prime_{i}$ and end at the neighbour $x_i$.  
Therefore, for $i \in [1, 2N-2]$ odd, the contribution given by the $i$th sub-walk is given by $G(x_{i-1}, x_{i})$, while for  $i \in [1, 2N-2]$ even or $i = 2N-1$, the contribution given by the $i$th sub-walk is given by
 $(dz\lambda )\, G(x_{i-1}, x_{i}^\prime)\cN_{x_{i}^\prime, x_{i-3}}(x_{i})$.
To get \eqref{eq:lemmaN}, we then neglect the
interaction between different sub-walks by upper bounding the factor  $\prod_{s^\prime t^\prime\in \cC(L)}(1+\lambda U_{s^\prime t^\prime})$ in \eqref{eq:PiN}
by $1$ for pairs $s^\prime, t^\prime$ such that $s^\prime$ and $t^\prime$ refer to steps belonging to different sub-walks.  
\end{proof}


\proof[Proof of \eqref{eq:bounds2} when $N\ge 2$.]
The goal is to upper bound (\ref{eq:lemmaN})
by a product of independent blocks then by taking the supremum over the points which are `seen' by some of the sub-walks in the sum.  The sums corresponding to each such block is then bounded from above by $\sB_z^\lambda$.  See Figure \ref{figUPblock}. 

\begin{figure}\centering
 \includegraphics[scale=1.4]{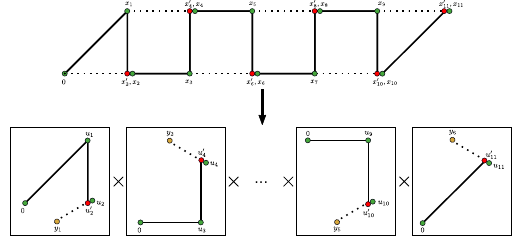}
   \caption{Top: Representation of \eqref{eq:lemmaN}. Bottom: Representation of  \eqref{eq:lemmaN_1}, with $N=6$. 
   }
 \label{figUPblock}
\end{figure}

For $i=1, \ldots, N-1$ we perform the following change of variables
\begin{align*}
u_{2i-1}=x_{2i-1}-x_{2 i-2}, \qquad
u_{2i}^\prime=x^\prime_{2 i}-x_{2 i-2}, \qquad
u_{2i}=x_{2i}-x_{2 i-2},
\end{align*} 
and  
\begin{align*}
u_{2N-1}^\prime= x_{2N-1}^\prime-x_{2N - 2}, \qquad u_{2N-1} = x_{2N-1}-x_{2N - 2},
\end{align*}
 so that, if 
${\bf u}^{N}= (u_1, u_2^\prime, u_2, u_3, u_4^{\prime}, u_4,
\ldots, u_{2N-2}^\prime, u_{2N-2}, u^\prime_{2N-1}, u_{2N-1})$,
then the translation invariance gives that the right hand side of \eqref{eq:lemmaN} equals
\begin{align}\label{eq:lemmaN_0}
& (z d \lambda)^N  \sum\limits_{  {\bf u}^{N}   \in (\mathbb{Z}^d)^{3N-1} }  \Bigg\{ 
 \prod_{i\in \{1, \ldots, N-1\} }
\Big(
G(0,u_{2i-1}) 
 G(u_{ 2 i -1 },u^\prime_{2 i})
 \mathcal{N}_{u^\prime_{2 i}, u_{2i-3}-u_{2i-2}}(u_{2i})
\Big)
 \\ \notag 
 & \hspace{4cm} \times G(0, u_{2N-1}^\prime) \, \,
 \mathcal{N}_{u_{2N-1}^\prime,u_{2N-3}-u_{2N - 2}}(u_{2N-1})
\Bigg\},
\end{align}
where we have used that $x_{2i-3}-x_{2i-2}=u_{2i-3}-u_{2i-2}$. 
We observe that $u_{2N-1}$ and $u_{2N-1}^\prime$ are present only in the last product. Therefore, we can first split the original sum and then take the sup over the last point seen by $u_{2N-1}$. We obtain that \eqref{eq:lemmaN_0} is bounded by
\begin{align}
& (z d \lambda)^N  \sum\limits_{  \tilde {\bf u}^{N}   \in (\mathbb{Z}^d)^{3N-1} }  \Bigg\{ 
 \prod_{i\in \{1, \ldots, N-1\} }
\Big(
G(0,u_{2i-1}) 
 G(u_{ 2 i -1 },u^\prime_{2 i})
 \mathcal{N}_{u^\prime_{2 i}, u_{2i-3}-u_{2i-2}}(u_{2i})
\Big)
 \\ \notag 
 & \hspace{4cm} \times  \sup_{y_{N} \in \mathbb{Z}^d} \sum\limits_{  u_{2N-1}^\prime, u_{2N-1}\in \mathbb Z^d }  G(0, u_{2N-1}^\prime) \, \,
 \mathcal{N}_{u_{2N-1}^\prime, y_N}(u_{2N-1})
\Bigg\},
\end{align}
where $\tilde{\bf u}^{N}:= (u_1, u_2^\prime, u_2, u_3, u_4^{\prime}, u_4,
\ldots, u_{2N-2}^\prime, u_{2N-2}$). We observe that the last term in the product is independent from the rest of the sum, see Figure \ref{figUPblock} for a graphic representation.

We iterate this strategy obtaining that the previous expression is bounded by 
\begin{align}
\label{eq:lemmaN_1}
&(z  d \lambda)^N  
 \prod_{i\in \{1, \ldots, N-1\} } \Bigg\{ \sup_{y_i \in \mathbb{Z}^{d}} 
 \sum\limits_{ u_{2i-1}, u_{2i}^\prime, u_{2i} \in \mathbb Z^d } \Big(
G(0,u_{2i-1}) 
 G(u_{ 2 i -1 },u^\prime_{2 i})
 \mathcal{N}_{u^\prime_{2 i}, y_i}(u_{2i})
\Big)\Bigg\}
\\ \notag
&\hspace{2.8cm} \times   
 \sup_{y_{N} \in \mathbb{Z}^d} \sum\limits_{  u_{2N-1}^\prime, u_{2N-1}\in \mathbb Z^d }   G(0, u_{2N-1}^\prime) \, \,
\mathcal{N}_{u_{2N-1}^\prime,y_N}(u_{2N-1}). 
\end{align}
Using that for all $v, y\in \mathbb Z^d$, $\sum_{u} \mathcal{N}_{v,y}(u)=\I{y}(v)$ (cf.  \eqref{eq:propertyN}), we obtain that \eqref{eq:lemmaN_1} equals
\begin{align}  \label{eq:lemmaN_2}
&(z d  \lambda)^N  
 \prod_{i\in \{1, \ldots, N-1\} } \Bigg\{ \sup_{y_i \in \mathbb{Z}^{d}} 
 \sum\limits_{ u_{2i-1}, u_{2i}^\prime  \in \mathbb Z^d } \Big(
G(0,u_{2i-1}) 
 G(u_{ 2 i -1 },u^\prime_{2 i})
 \I{y_i}(u_{2i}^\prime)
\Big)\Bigg\}
\\ \notag
&\hspace{2.8cm} \times   
 \sup_{y_{N} \in \mathbb{Z}^d} \sum\limits_{  u_{2N-1}^\prime \in \mathbb Z^d }   G(0, u_{2N-1}^\prime)
\I{y_N}(u_{2N-1}^\prime),
\end{align}
The result now follows by observing that each term inside the curly brackets is equal to  $\sB^\lambda_z$, while the last term is smaller than $\sB^\lambda_z$.
\qed

 
We are now ready to  prove \eqref{eq:bounds2b}.

\proof[Proof of  \eqref{eq:bounds2b} when $N\geq 2$.]
We start again from \eqref{eq:PiN} and use the same decomposition
of the prudent walk $w$ into $2N-1$ prudent sub-walks $w^1\in\Wa(0,x_1)$, 
$\ldots$, $w^{2N-1}\in\Wa(x_{2N - 2}, x_{2N-1})$, with $x_{2N-1}=x$, which was presented in the proof of Lemma \ref{lemma:lemmaN}.
Given such a decomposition and the sequence of $3N-1$ points
$ x_1, x_2, x_2^{\prime}$, $x_3$, \ldots, $x_{2N-1} = x$, we use the upper bound \eqref{eq:inequalitycos} obtaining 
\begin{equation}\label{eq:cosboundN}
\begin{split}
1 - \cos (k \cdot x)  \leq & \;
(3N - 1) \,  \Bigg [ \sum\limits_{i=0}^{N-2} \big((1 - \cos ( k \cdot  (x_{2i+1}-x_{2i}) )\big) 
\\ & +
\sum\limits_{i=1}^{N-1} \big(1 - \cos ( k \cdot (  x_{2i}^\prime - x_{2i-1})) \big)
+ \big(1 - \cos ( k \cdot (  x_{2N-1}^\prime - x_{2N-2}) ) \big)
\\ & +
\sum_{i=1}^{N-1} \big(1 - \cos (k \cdot  (x_{2i} - x_{2i}^{\prime}   ))\big)
+ \big(1 - \cos (k \cdot  (x_{2N-1} - x_{2N-1}^{\prime}   )) \big)
\Bigg ],
\end{split}
\end{equation}
where we recall the convention $x_{0}=0$.
We can thus ``distribute'' the displacement term $1 - \cos (k \cdot x)$ along the solid lines in Figures \ref{figN}, where the first summand corresponds to the horizontal lines, the second summand corresponds to the vertical lines, and the third summands accounts for the single edges between red and green marked vertices.
Hence, the modification of the upper bound  \eqref{eq:bounds2} due to the factor  $1 - \cos (k \cdot x)$ is simply 
\begin{enumerate}[(a)]
\item 
to replace one of the 
functions $G(x_{2i-2}, x_{2i-1})$
 by 
$G(x_{2i-2}, x_{2i-1}) \big ( 1 - \cos ( k \cdot ( x_{2i-1} -x_{2i-2}  )  )  \big )$, for $i=1, \ldots, N$;
\item 
to replace one of the 
functions
$G(x_{2i-1}, x^\prime_{2i})$
 by 
 $G(x_{2i-1}, x^\prime_{2i}) \big ( 1 - \cos ( k \cdot (  x^\prime_{2i} -x_{2i-1}   )  \big )$, for $i=1, \ldots, N-1$,
 or $G(x_{2N-2}, x^\prime_{2N-1})$
 by 
 $G(x_{2N-2}, x^\prime_{2N-1}) \big ( 1 - \cos ( k \cdot (  x^\prime_{2N-1} -x_{2N-2}   )  \big )$;
 \item to replace one of the 
functions
$ \mathcal{N}_{x^\prime_{2 i}, x_{2i-3}}(x_{2i})$
by $ \mathcal{N}_{x^\prime_{2 i}, x_{2i-3}}(x_{2i}) (1 - \cos ( k \cdot( x_{2i} - x_{2i}^{\prime}  )  ) )$, for $i=1, \ldots, N-1$, or 
$\mathcal{N}_{x_{2N-1}^\prime,x_{2N-3}}(x_{2N-1})$
by
 $\mathcal{N}_{x_{2N-1}^\prime,x_{2N-3}}(x_{2N-1}) (1 - \cos ( k\cdot ( x_{2N-1} - x_{2N-1}^\prime )  ) )$. 
 \end{enumerate}
 We then consider separately the cases (a), (b) and (c).
 
\paragraph{Case (a).} 
 We consider the case where for a $j \in \{1,2, \ldots, N-1\}$ the function $G(x_{2j-2}, x_{2j-1})$ in the right-hand side of \eqref{eq:lemmaN}
 is replaced by $\big (1 - \cos (k  \cdot ( x_{2j-1} - x_{2j-2}  )  ) \big ) G(x_{j-2}, x_{2j-1})$. The case $j=N$ is easier and follows similarly, we omit the details.
\begin{figure}
\centering
 \includegraphics[scale=1.4]{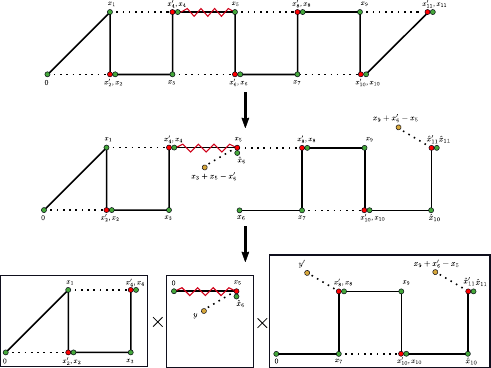}
 \caption{Diagrammatic decomposition in case (a) with $N= 6$ and $j= 3$. 
 In the middle the diagrammatic decomposition after the change of variables \eqref{eq:case_a2} and at the bottom \eqref{eq:case_asup}.}
 \label{figNa}
\end{figure}
This replacement is bounded by (see Figure \ref{figNa}),
\begin{align}\label{eq:case_a}
(d z \lambda)^N \, \sum\limits_{ \substack{  x_{2j-2}, x_{2j-3}\in \mathbb Z^d \\ \underline{\bf x}^n \in (\mathbb{Z}^d)^{ 3(N-j)+2 }}}& L^{(j-1)}(x_{2j-3},x_{2j-2}) 
\big (1 - \cos (k  \cdot ( x_{2j-1} - x_{2j-2}  )  ) \big ) G(x_{2 j-2}, x_{2j-1} )
\\ \notag
&  \times G(x_{2 j-1}, x_{2j}^\prime )
\mathcal{N}_{x^\prime_{2 j}, x_{2j-3}}(x_{2j})
 \\ \notag
& \qquad \times \prod_{i\in \{j+1, \ldots, N-1\} }
\Big\{
G(x_{2 i-2},x_{2i-1}) 
 G(x_{ 2 i -1 },x^\prime_{2 i})
 \mathcal{N}_{x^\prime_{2 i}, x_{2i-3}}(x_{2i})
\Big\}
 \\ \notag & \qquad \qquad \times G(x_{2N - 2}, x_{2N-1}^\prime) \, \,
 \mathcal{N}_{x_{2N-1}^\prime,x_{2N-3}}(x_{2N-1}),
\end{align}
where
$  \underline{\bf x}^{n}  = ( x_{2j-1}, x_{2j},$ $ x_{2j}^\prime, x_{2j+1}, $ $ x_{2j+2}^\prime, $ $  x_{2j+2}, x_{2j+3},$ $  \ldots, x_{2N-3}, x_{2N-2}^\prime, $ $  x_{2N-2},  x_{2N-1}^\prime,  x_{2N-1}) $ and 
\begin{align}
\notag
L^{(n)}(x,y) & : =
 \sum\limits_{  \substack{  \tilde{\bf x}^{n}  \in (\mathbb{Z}^d)^{3n-3} \\ y^\prime \in \mathbb{Z}^d  }}  
 \Bigg\{ 
 \prod_{i\in \{1, \ldots, n-1\} }
\Big(
G(x_{2 i-2},x_{2i-1}) 
 G(x_{ 2 i -1 },x^\prime_{2 i})
\mathcal{N}_{x^\prime_{2 i}, x_{2i-3}}(x_{2i})
\Big)
 \\ \label{eq:definitionLN}
 & \hspace{4cm} \times  G(x_{2n-2}, x    ) G(x, y^\prime)  \,
 \mathcal{N}_{y^\prime, x_{2n-3}}(y)
\Bigg\}.
\end{align}
Let us note that with the same argument used to upper bound \eqref{eq:lemmaN}, for any $n\ge 1$,
\begin{align}
\label{eq:auxiliary1}
\sum_{x,y\in \mathbb Z^d} L^{(n)}(x,y) & \le   (\sB_z^\lambda)^n\, .
\end{align}

Our goal is to obtain the diagram represented at the bottom of Figure \ref{figNa} with a suitable change of variables. For this purpose we use translation invariance to replace 
\begin{itemize}[$*$]
\item $\mathcal{N}_{x^\prime_{2 j}, x_{2j-3}}(x_{2j})$  by
$\mathcal{N}_{x_{2 j-1}, x_{2j-3} + x_{2j-1} - x_{2j}^\prime}(x_{2j}+x_{2j-1}-x_{2j}^\prime),$
\item $
G(x_{2j-1}, x_{2j}^\prime)
  \mbox{ by } 
G(x_{2N-2}, x_{2j}^\prime-x_{2j-1}+x_{2N-2}),$
\item 
$G(x_{2N-2}, x_{2N-1}^\prime)
  \mbox{ by } 
G( x_{2j}^\prime-x_{2j-1}+x_{2N-2}, x_{2j}^\prime-x_{2j-1}+x_{2N-1}^\prime)$,
\item 
$ \mathcal{N}_{x^\prime_{2N-1}, x_{2N-3}}(x_{2N-1}) $
by
$ \mathcal{N}_{x_{2j}^\prime-x_{2j-1}+x_{2N-1}^\prime, x_{2N-3}+x_{2j}^\prime-x_{2j-1}}(x_{2j}^\prime-x_{2j-1}+x_{2N-1}) $
\end{itemize}
in the right-hand side of \eqref{eq:case_a}.
Moreover, we make a change of variables defining
\begin{multicols}{2}
\begin{itemize}[$*$]
     \item  $\hat{x}_{2j} := x_{2j} + x_{2j-1} - x_{2j}^\prime$,
     \item $\hat x_{2N-1}:=x_{2j}^\prime-x_{2j-1}+x_{2N-1}$,
     \item $\hat x_{2N-1}^\prime:=x_{2j}^\prime-x_{2j-1}+x_{2N-1}^\prime$,
     \item $\hat x_{2N-2}:=x_{2j}^\prime-x_{2j-1}+x_{2N-2}$,
\end{itemize}
\end{multicols}
\noindent thus replacing the sum over
$x_{2j}$,
$x_{2j}^\prime$,
$x_{2N-1}$,
$x_{2N-1}^\prime$
by the sum over 
$\hat{x}_{2j}$
$\hat{x}_{2j}^\prime$,
$\hat{x}_{2N-1}$,
$\hat{x}_{2N-2}$.
This leads to 
\begin{equation}\label{eq:case_a2}
\begin{split}
(d z \lambda)^N \, & \sum\limits_{ \substack{  x_{2j-2}, x_{2j-3}\in \mathbb Z^d \\ \underline{\bf x}^n \in (\mathbb{Z}^d)^{ 3(N-j)+2 }}}  L^{(j-1)}(x_{2j-3},x_{2j-2}) 
\\ 
& \times
\big (1 - \cos (k  \cdot ( x_{2j-1} - x_{2j-2}  )  ) \big ) G(x_{2 j-2}, x_{2j-1} )
\mathcal{N}_{x_{2 j-1}, x_{2j-3}+x_{2j-1} - x_{2j}^\prime}( \hat{x}_{2j} )
\\ &\qquad
 \prod_{i\in \{j+1, \ldots, N-1\} }
\Big\{
G(x_{2 i-2},x_{2i-1}) 
 G(x_{ 2 i -1 },x^\prime_{2 i})
 \mathcal{N}_{x^\prime_{2 i}, x_{2i-3}}(x_{2i})
\Big\}
 \\  & \qquad \qquad \times 
G(x_{2N-2},  \hat{x}_{2N-2}) G(  \hat{x}_{2N-2},  \hat{x}_{2N-1}^\prime) \, \,
 \mathcal{N}_{ \hat{x}_{2N-1}^\prime, x_{2N-3}+x_{2j}^\prime-x_{2j-1}}( \hat{x}_{2N-1}),
\end{split}
\end{equation}
where 
we denote by $\underline{\bf x}^n$
the vector of variables over which we sum.
We refer to the middle of Figure \ref{figNa} for a graphic representation of the previous sum.
Note that the old variables  $x_{2j}$ and $x_{2j}^\prime$ appearing in the sum are interpreted as a deterministic function of the variables introduced after the change of coordinates, over which we sum.
The previous expression can be bounded from above by 
\begin{equation}\label{eq:case_asup}
\begin{split}
(d z \lambda)^N \,  
& \sum\limits_{ x_{2j-2}, x_{2j-3}\in \mathbb Z^d}  L^{(j-1)}(x_{2j-3},x_{2j-2}) \\ 
& \times
\sup_{y} \Bigg \{ \sum\limits_{x_{2j-1}, \hat{x}_{2j}\in \mathbb Z^d}
\big (1 - \cos (k  \cdot ( x_{2j-1} - x_{2j-2}  )  ) \big ) G(x_{2 j-2}, x_{2j-1} )
\mathcal{N}_{x_{2 j-1},y}(\hat x_{2j}) 
 \\ & 
\times
 \sup_{y^\prime} \Big\{
 \sum\limits_{ \tilde{\bf x}^{n}\in (\mathbb Z^d)^{3(N-j)-3}}
 G(x_{2j},x_{2j+1}) 
 G(x_{ 2 j+1 },x^\prime_{2 j+2})
\mathcal{N}_{x^\prime_{2j+2}, y^\prime}(x_{2j+2})
 \\&\qquad \qquad \times
 \prod_{i\in \{j+2, \ldots, N-1\} }
\big\{
G(x_{2i-2},x_{2i-1}) 
 G(x_{ 2 i -1 },x^\prime_{2 i})
\mathcal{N}_{x^\prime_{2 i}, x_{2i-3}}(x_{2i})
\big\}
 \\ & \qquad \qquad \times 
 \sum\limits_{\hat{x}_{2N-2}, \hat{x}_{2N-1}^\prime, \hat{x}_{2N-1}   }
G(x_{2N-2},  \hat{x}_{2N-2}) G(  \hat{x}_{2N-2},  \hat{x}_{2N-1}^\prime) \, 
\\  & \qquad\qquad \times 
 \mathcal{N}_{ \hat{x}_{2N-1}^\prime, x_{2N-3}+x_{2j}^\prime-x_{2j-1}}( \hat{x}_{2N-1})
 \Big \} \Bigg \},
\end{split}
\end{equation}
where $ \tilde{\bf x}^{n}  = ( x_{2j+1}, $ $ x_{2j+2}^\prime, $ $  x_{2j+2}, x_{2j+3},$ $  \ldots, x_{2N-3}, x_{2N-2}^\prime, $ $  x_{2N-2}) $.
By translation invariance we note that the supremum over $y^\prime$ does not depend on $x_{2j}$ and the supremum over $y$ does not depend on $x_{2j-2}$, which can then be replaced both by $0$, see the bottom of Figure \ref{figNa}. 
Let us note that the last term in the sum that depends on the supremum over $y^\prime$ is independent of the rest of the sum, so that it can be bounded from above by $(\sB_z^\lambda)^{N-j}$
by using the same strategy as for the upper bound to (\ref{eq:lemmaN}). To be more precise, the previous expression can be bounded from above by
\begin{align}\label{eq:case_asup2}
\sup_{y^\prime} & \big \{\sum\limits_{x_{2j+1}, x_{2j+2}^\prime, x_{2j+2} }
G(0,x_{2j+1}) 
 G(x_{ 2 j +1 },x^\prime_{2 j+2})
 \mathcal{N}_{x^\prime_{2 j+2}, y^\prime}(x_{2j+2}) \big \}
 \\ \notag & \times 
 \prod_{i\in \{j+2, \ldots, N-1\} }
\sup_{y_i}\big\{
\sum\limits_{x_{2i-1}, x_{2i}^\prime, x_{2i} }
G(x_{2i-2},x_{2i-1}) 
 G(x_{ 2 i -1 },x^\prime_{2 i})
 \mathcal{N}_{x^\prime_{2 i}, y_i}(x_{2i}) \big \}
\\ \notag
&\times 
\sup_{y_N} \big \{
\sum\limits_{\hat{x}_{2N-2}, \hat{x}_{2N-1}^\prime, \hat{x}_{2N-1}   }
G(x_{2N-2},  \hat{x}_{2N-2}) G(  \hat{x}_{2N-2},  \hat{x}_{2N-1}^\prime) \, 
 \mathcal{N}_{ \hat{x}_{2N-1}^\prime,y_N}( \hat{x}_{2N-1}) \big  \}
 \Big \}\\ \notag
&= (\sB_z^\lambda)^{N-j},
\end{align}
where for the last identity we have used the translation invariance and (\ref{eq:propertyN}) to observe that each supremum in the previous expression equals $\sB_z^\lambda$.
For the remaining term, we use again (\ref{eq:propertyN}) 
to obtain that 
\begin{align*}
&    \sup\limits_{ y \in \mathbb{Z}^d}
  \sum\limits_{ \substack{ u, v \in \mathbb{Z}^d }}
\big (1 - \cos (k  \cdot u  ) \big ) \, 
G(0, u  )  \, 
 \mathcal{N}_{u,   y}(
v  )   
= \|  \sY_{z,k}^\lambda * \1  \|_\infty 
\end{align*}
and we use (\ref{eq:auxiliary1}) to bound the remaining part of the sum by $(\sB_z^\lambda)^{j-1}$.
Hence from this we conclude  
that \eqref{eq:case_a}
is less or equal than
 $(dz \lambda)^N
  \,   \|  \sY_{z,k}^\lambda * \1  \|_\infty 
    (\sB^{\lambda}_z)^{N-1}.$
    This concludes case ({\rm a}).

 \paragraph{Case (b).}
We  consider the case in which one of the functions $G(x_{2j-1}, x^\prime_{2j})$ in the right-hand side of \eqref{eq:lemmaN} is replaced  by 
 $G(x_{2j-1}, x^\prime_{2j}) \big ( 1 - \cos ( k \cdot (  x^\prime_{2j} -x_{2j-1}   )  \big )$
 for some $j \in \{1, \ldots, N-1\}$. 
 The case in which $G(x_{2N-2}, x^\prime_{2N-1})$ is replaced by
 $G(x_{2N-2}, x^\prime_{2N-1}) \big ( 1 - \cos ( k \cdot (  x^\prime_{2N-1} -x_{2N-2}   )  \big )$ follows similarly. 
  \begin{figure}[t]
\centering
 \includegraphics[scale=1.4]{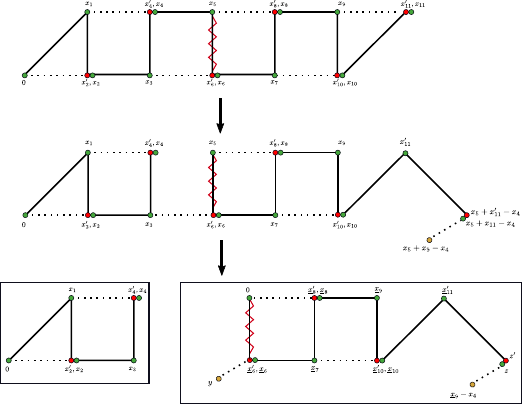}
 \caption{Top : Diagrammatic decomposition in case (b), with $j=3$ and $N=6$.
Middle and bottom : \eqref{eq:b-c} and \eqref{eq:b-b}, respectively.
 }
 \label{figNb}
\end{figure}
This replacement leads to the summation, illustrated in  
 Figure \ref{figNb},
 \begin{equation}
 \label{eq:b-c}
 \begin{split}
(d z \lambda)^N \,& \sum\limits_{ \substack{  x_{2j-2}, x_{2j-3}\in \mathbb Z^d \\ \underline{\bf x}^n \in (\mathbb{Z}^d)^{ 3(N-j)+2 }}}  L^{(j-1)}(x_{2j-3},x_{2j-2})
\\  & \qquad  \times G(x_{2 j-2}, x_{2j-1} ) 
G(x_{2j-1}, x_{2j}^\prime) \big (1 - \cos (k ( x_{2j}^\prime - x_{2j-1}  )  ) \big )  
\mathcal{N}_{x^\prime_{2 j}, x_{2j-3}}(x_{2j}) \\
& \qquad \times \prod_{i\in \{j+1, \ldots, N-1\} }
\Big\{
G(x_{2 i-2},x_{2i-1}) 
 G(x_{ 2 i -1 },x^\prime_{2 i})
\mathcal{N}_{x^\prime_{2 i}, x_{2i-3}}(x_{2i})
\Big\}
 \\ & \qquad \times G(x_{2N - 2}, x_{2N-1}^\prime) \, \,
\mathcal{N}_{x_{2N-1}^\prime,x_{2N-3}}(x_{2N-1})
\end{split}
\end{equation}
 where
  $  \underline{\bf x}^{n}  = ( x_{2j-1}, x_{2j},$ $ x_{2j}^\prime, x_{2j+1}, $ $ x_{2j+2}^\prime, $ $  x_{2j+2}, x_{2j+3},$ $  \ldots, x_{2N-3}, x_{2N-2}^\prime, $ $  x_{2N-2},  x_{2N-1}^\prime,  x_{2N-1}) $.
We now use translation invariance and we replace,
\begin{itemize}[$*$]
\item  $G(x_{2j-2}, x_{2j-1})$ by $G(x^\prime_{2N-1}, x_{2j-1} +  x^\prime_{2N-1}   - x_{2j-2} )$, and
\item  $\mathcal{N}_{x_{2N-1}^\prime, \, x_{2N-3}}(x_{2N-1})$ by $ \mathcal{N}_{x_{2N-1}^\prime +  x_{2j-1} - x_{2j-2},\, x_{2N-3} +  x_{2j-1} - x_{2j-2}}(x_{2N-1} +  x_{2j-1} - x_{2j-2})$
\end{itemize}
in the previous expression.
In such a way,  we obtain a new sum which corresponds
to the middle diagram in Figure \ref{figNb}.
After this replacement we  perform a change of variables, i.e, we define
$\underline{x}_i = x_i + x_{2j-1}$
and $\underline{x}_i^\prime = x_i^\prime + x_{2j-1}$  for each $x_i$ and $x_i^\prime$ in the sum such that $i \geq 2j$, 
 $z= \underline{x}_{2N-1}+ x_{2j-1} - x_{2j-2} $ and
 $z^\prime= \underline{x}^\prime_{2N-1}+ x_{2j-1} - x_{2j-2} $
 and then obtain the following expression,
 \begin{equation}
 \label{eq:b-b}
 \begin{split}
(d z \lambda)^N &\, \sum\limits_{  x_{2j-3}, x_{2j-2}\in \mathbb Z^d}  L^{j-1}(x_{2j-3},x_{2j-2}) \\ 
&\times
\sup_{y\in \mathbb Z^d}\Bigg\{  \sum\limits_{  \underline{\bf x}^n \in (\mathbb{Z}^d)^{ 3(N-j)+2 }}  G(0, \underline{x}_{2j}^\prime)  \big (1 - \cos (k  \cdot  \underline{x}_{2j}^\prime    ) \big ) 
\mathcal{N}_{\underline{x}^\prime_{2 j},\, y}(\underline{x}_{2j}) \\
 & \times \prod_{i\in \{j+1, \ldots, N-1\} }
\Big\{
G(\underline{x}_{2 i-2},\underline{x}_{2i-1}) 
 G(\underline{x}_{ 2 i -1 },\underline{x}^\prime_{2 i})
 \mathcal{N}_{\underline{x}^\prime_{2 i}, \, \underline{x}_{2i-3}}( \underline{x}_{2i})
\Big\}
 \\ &  \times G(\underline{x}_{2N - 2}, \underline{x}_{2N-1}^\prime) \, \,G(\underline{x}^\prime_{2 N-1}, z^\prime ) 
\mathcal{N}_{z^\prime,\, \underline{x}_{2N-3} - {x}_{2j-2}}(z)\Bigg\} \, ,
\end{split}
\end{equation}
which corresponds to the bottom diagram in Figure \ref{figNb}.
In the previous sum we denoted by $\underline{\boldsymbol{x}}^n$ the vector corresponding 
to the set of  variables other than $x_{2j-3}$ and $x_{2j-2}$ which appear in the summand.

We now proceed similarly to the proof of \eqref{eq:bounds2}, 
 using the same argument as in \eqref{eq:case_asup}--\eqref{eq:case_asup2}
 to obtain that the expression inside the big curly brackets 
 is bounded from above by 
$\| \sY_{z,k}^\lambda * \1   \|_{\infty}(\sB_z^\lambda)^{N-j}$. 
We then use \eqref{eq:auxiliary1} to conclude case ({\rm b}).


 \paragraph{Case (c).} 
 We now consider the case
 in which the function 
$ \mathcal{N}_{x^\prime_{2 j}, x_{2j-3}}(x_{2j})$
in the right-hand side of
 \eqref{eq:lemmaN} is replaced by $\mathcal{N}_{x^\prime_{2 j}, x_{2j-3}}(x_{2j}) (1 - \cos ( k  \cdot ( x_{2j} - x_{2j}^{\prime}  )  ) $
for some $j=1, \ldots, N-1$. The  remaining case, in which 
$\mathcal{N}_{x_{2N-1}^\prime,x_{2N-3}}(x_{2N-1})$
is replaced by
 $\mathcal{N}_{x_{2N-1}^\prime,x_{2N-3}}(x_{2N-1}) (1 - \cos ( k \cdot ( x_{2N-1}^\prime - x_{2N-1} )  ) )$, follows analogously.
 This replacement leads to (cf. \eqref{eq:lemmaN}),
 \begin{equation}
 \label{eq:dec-(c)}
 \begin{split}
& (z d \lambda)^N  \sum\limits_{  {\bf x}^{N}   \in (\mathbb{Z}^d)^{3N-1} }  \Bigg\{ 
 \prod_{i\in \{1, \ldots, N-1\} }
\Big(
G(x_{2 i-2},x_{2i-1}) 
 G(x_{ 2 i -1 },x^\prime_{2 i})
\mathcal{N}_{x^\prime_{2 i}, x_{2i-3}}(x_{2i})
\Big)
 \\ & \hspace{2cm} \times G(x_{2N - 2}, x_{2N-1}^\prime) \, \,
 \mathcal{N}_{x_{2N-1}^\prime,x_{2N-3}}(x_{2N-1})
\Bigg\} \big(1 - \cos ( k \cdot ( x_{2j} - x_{2j}^{\prime}  ))\big).
\end{split}
\end{equation}
 We note that, from the definition of 
 $\mathcal{N}_{x^\prime_{2 j}, x_{2j-3}}(x_{2j})$,
 it is necessarily the case that 
$ x_{2j}$ is a neighbour of $x_{2j}^{\prime}$
(the sum equals $0$ otherwise).
Note also that, by symmetry, 
 the value of the whole sum restricted to the case 
$x_{2j} - x_{2j}^\prime = e$ for some unit vector $e \in \mathbb{Z}^d$ 
gives $C \, \big  (1 - \cos (k \cdot e  ) \big ) $,
for some constant $C$ which does not depend on $e$.
Hence, 
using the notation $k= (k_1, \ldots, k_d)$,
the whole 
expression equals,
\begin{equation*}
 \begin{split}
&  (z d \lambda)^N  \sum\limits_{  {\bf x}^{N}   \in (\mathbb{Z}^d)^{2N-1} }  \Bigg\{ 
 \prod_{i\in \{1, \ldots, N-1\} }
\Big(
G(x_{2 i-2},x_{2i-1}) 
 G(x_{ 2 i -1 },x^\prime_{2 i})
\mathcal{N}_{x^\prime_{2 i}, x_{2i-3}}(x_{2i})
\Big)
 \\ & \hspace{2cm} \times G(x_{2N - 2}, x_{2N-1}^\prime) \, \,
\mathcal{N}_{x_{2N-1}^\prime,x_{2N-3}}(x_{2N-1})
\Bigg\}  \frac{1}{d} \Bigg (\sum\limits_{i=1}^d \big (1 - \cos ( k_i   )  \big )  \Bigg ) .
\end{split}
 \end{equation*} 
By observing that $1-\hat D(k)=\frac1d\sum_{i=1}^d \big (1 - \cos ( k_i   )  \big )$ is non-negative, reasoning as in \eqref{eq:lemmaN_0} and \eqref{eq:lemmaN_1} we obtain that \eqref{eq:dec-(c)} is less or equal than
 $ (d z\lambda )^N \,  ( \sB_z^\lambda)^{N} \, (1-\hat D(k)),$
 which concludes case ({\rm c}).
 \qed

\section{Convergence of the lace expansion and bootstrap argument}\label{sec:bootstrap}
The goal of this section is to prove  
Theorem \ref{thm:BubbleCond}, which states the finiteness of the critical prudent bubble diagram in dimensions $d  > 5$ for the weakly prudent walk with parameter
$\lambda$ small enough or for the 
 strictly prudent walk in large enough dimensions.

\subsection{Bounds under the bootstrap assumption} 
Since the indicator function $\1$, which was introduced in \eqref{ourindicator}, is not integrable, it is not suited for a Fourier analysis. We therefore introduce the following approximation: for $R \in \mathbb{R}_{\geq 0}$ and $x \in \mathbb{Z}^d$ we define
 \begin{equation}\label{def:1-Rapprox}
 \IR (x) :=
 \begin{cases}
  \frac{1}{d} e^{ - \frac{x^2  }{R} } & \mbox{ if $x \bot  0$ and $x \neq 0$,} \\
 0 & \mbox{ otherwise,}
 \end{cases}
 \end{equation}
 where $x\bot 0$ was defined in \eqref{def:abotx}.
 Let us observe that $\lim_{R\to +\infty} \IR=\1$ pointwise.
We also define, for any $i \in \{1,\ldots, d\}$,
\begin{align}
\IR^i(x) : = 
 \begin{cases}
  \frac{1}{d} e^{ - \frac{x^2  }{R} } & \mbox{ if $x \in \boldsymbol{e}_i \mathbb{Z}$} \\
 0 & \mbox{ otherwise}
 \end{cases}
\end{align}
where $ \boldsymbol{e}_i $ is the $i$-th vector of the canonical base of $\mathbb R^d$.

Observe that $ \IR (x) :=  \sum_{i=1}^{d} \IR^i(x) - \delta_0$. In the next result we show that the Fourier transform of $\IR(x)+\delta_0$ is close to a Gaussian function.

\begin{lemma}\label{lem:Fourierindicatori}
There exists two positive constants $c_1,C_2>0$ independent of the dimension such that for any $i=1, \ldots, d$ and $k = (k_1, \ldots, k_d) \in [-\pi, \pi]^d$, 
\begin{equation}\label{eq:Fourierindicatori}
\frac{c_1\sqrt{ R} }{d}
e^{  - \frac{R}{2} \, k_i^2  }\le \hIR^i (k) \le 
\frac{C_2\sqrt{R} }{d}
e^{  - \frac{R}{2} \, k_i^2  }.
\end{equation}
\end{lemma}

\begin{proof}
Since $\hIR^i (k) = 
\frac 1d\sum_{w \in \mathbb{Z}} 
e^{- \frac{w^2}{R}} \, e^{i w k_i  }$ for any $k = (k_1, \ldots, k_d) \in [-\pi, \pi]^d$ and $i=1, \ldots, d$,
we fix $t\in \mathbb R$ and we consider
\begin{align}
\label{eq:fourier_tr}
\sum\limits_{w \in \mathbb{Z}} 
e^{- \frac{w^2}{R}} \, e^{i w t  }&=
\sum\limits_{w \in \mathbb{Z}} \sum\limits_{n = 0}^{+\infty}
\frac{(iwt)^n}{n!}e^{- \frac{w^2}{R}} 
= \sum\limits_{n = 0}^{+\infty} \frac{(it)^{2n }}{(2n)!}\sum\limits_{w \in \mathbb{Z}} w^{2n}e^{- \frac{w^2}{R}}, 
\end{align}
where we have used that $ \sum_{w \in \mathbb Z} w^ne^{- \frac{w^2}{R}} =0$ if $n\in 2\mathbb N+1$.
We note that there exist two positive constants $c,C>0$ such that 
\begin{align}\label{eq:sum_int_com}
 c\int_{\mathbb{R}} x^{2n}e^{- \frac{x^2}{R}} \dd x\le 
 \sum\limits_{w \in \mathbb{Z}} w^{2n}e^{- \frac{w^2}{R}} \le C\sqrt{R} (Rn)^{n}e^{-n} .
\end{align}
We then observe that 
\[
\int_{\mathbb R} x^{2n}e^{- \frac{x^2}{R}} \dd x =\sqrt{2\pi R}\Big(\frac{R}{2}\Big)^n \frac{(2n)!}{2^n n!},
 \]
and Stirling's formula gives
$
(Rn)^{n}e^{-n}\sim \frac{1}{\sqrt{2}}\Big(\frac{R}{2}\Big)^n \frac{(2n)!}{2^n n!}.
$
This concludes the proof. \end{proof}

\begin{remark}\label{rem:1pos}
It follows from \eqref{eq:Fourierindicatori} and from  $\IR (x) :=  \sum_{i=1}^{d} \IR^i(x) - \delta_0(x)$ that  the Fourier transform of 
\begin{equation}\label{def:IR*}
\IR^*  : = \IR  + \delta_0.
\end{equation}
is finite and positive for every   $R \in (0, \infty)$. 
We use this property in the proof of Lemma \ref{lem:equi510} below.
\end{remark}
For $k \in [-\pi, \pi]^d$, we define the operator $\Delta_k$ as
\begin{equation}
- \frac{1}{2} \Delta_k \hat{f}(\ell) = \hat{f}(\ell) \, - \, 
\frac{1}{2} \Big (  \hat{f}(\ell+k)  + \hat{f}(\ell-k)   \Big )
\end{equation}
For any $z \in [0, z_c)$, let $p(z)\in [0,\frac 1{2d})$ be the point such that $ \hat C_{p(z)}(0)=\hat G_z(0)$. 
We then define:
\begin{equation}
\begin{split}
U_{p(z)}( k,\ell) : =  \hat{C}_{p(z)}(k)^{-1} \Big(  \hat C_{p(z)}(\ell-k) & \hat C_{p(z)}(\ell)\\ &+
 \hat C_{p(z)}(k+\ell)  \hat C_{p(z)}(\ell)+ \hat C_{p(z)}(\ell-k) \hat C_{p(z)}(k+\ell)\Big)
 \end{split}
\end{equation}
and
\begin{equation} \label{eq:bootstr}
\left.\begin{aligned} f_1(z) & : =  2 \, d \, z \\
f_2(z) & : = 
\sup_{k \in [-\pi, \pi]^d} \frac{ |\hat{G}_z(k)|   }{ \hat{C}_{p(z)}(k)} \\
f_3(z) & : =  \sup\limits_{k,\ell \in [-\pi, \pi]^d} \frac{ \frac{1}{2} \Big |  \Delta_k \hat{G}_z(\ell)   \Big |     }{ U_{p(z)}(k,\ell)  }
\end{aligned}\right\}
\Longrightarrow f(z) := \max\big\{ f_1(z), f_2(z), f_3(z)   \big\}.
\end{equation}
Our goal is to show that,  if the dimension is large enough or the dimension is greater than $5$ and $\lambda$ is small enough and  we assume that $f(z) \leq  4$
for any $z \in (0, z_c)$ (the so-called \textit{bootstrap assumption}), it is actually the case that $f(z) \leq 2$. 
A first step towards a proof of such claim is the next lemma. 
\begin{lemma}\label{lem:equi510}
Suppose that $d  > 5$.  There exists an absolute  finite constant $\underline{c}$ (which is  independent from $z$,  $\lambda$ and the dimension $d$) such that, for any $z \ge 0$,
if  $f(z) \leq  K$ for some $K \in (0, \infty)$,  then
\begin{align}
\sB^\lambda_z =   \| G^\lambda_z   * G^\lambda_z  * \1    \|_{\infty} &   \leq K^2
 \frac{  \, \underline{c}}{d}, \label{eq:2}\\
\| \sY^\lambda_{z,k} *  \1 \|_{\infty} & \leq    \frac{3 \,  \underline{c} \, K \,  }{\hat C_{p(z)}(k)},  \label{eq:3}
 \end{align}
 where  $\sY^\lambda_{z,k}$ was defined in \eqref{eq:defHzk}.
 \end{lemma}
 
In order to prove the lemma, we first prove a preliminary estimate. 
\begin{lemma}\label{lem:RWbubble}
Suppose $d>5$. There exists a constanct $c>0$ such that if $f(z)<K$ for some $K>0$, then for all $R>0$, 
\begin{equation} \label{eq:DDGG1}
	D * D * G_z* G_z   * \IR^* (x) \le c K^2/d. 
\end{equation}
\end{lemma}
\begin{proof}
Using the assumption that $f \leq K$ and that $ \IR^*$ has a positive Fourier transform (recall Remark \ref{rem:1pos}), we get
 \begin{align}
 \notag
D * D * G_z* G_z   * \IR^* (x) &  = 
\frac{1}{(2 \pi)^d} \int_{ [- \pi, \pi]^d  } \, \dd k  \, e^{  i k x  }  \hat{D}(k)^2 \hat{G}_z(k)^2  \hIR^*  (k)  
\\
\notag
& \leq 
\frac{1}{(2 \pi)^d} \int_{ [- \pi, \pi]^d  } \, \dd k \,   \hat{D}^2 (k) \hat{G}_z^2(k) \hIR^*  (k)    \\
 \notag
& \leq 
K^2 \frac{1}{(2 \pi)^d} \int_{ [- \pi, \pi]^d  } \, \dd k \, \hat{D}^2 (k) \hat{C}_{p(z)}^2(k) \hIR^* (k) \\
 \notag
& =
K^2 D * D * C_{p(z)} * C_{p(z)} * \IR^*(0)\\
 \notag
& \leq 
K^2 D * D * C_{\frac1{2d}} * C_{\frac1{2d}} * \IR^*(0)\\
\label{eq:lastIntDDGG1}
& = 
K^2 \frac{1}{(2 \pi)^d} \int_{ [- \pi, \pi]^d  } \, \dd k \, \hat{D}^2 (k) \hat{C}_{\frac{1}{2d}}^2(k) \hIR^* (k) .
\end{align}
By using that $\hat D(k)=\frac1d \sum_{i=1}^d \cos(k_i)$ and the fact that $\hat C_{\frac1{2d}}(k)=\frac{1}{1-\hat D(k)}$, first we can replace $\hIR^* (k)$ by  $ C_2\sqrt{R}  e^{ - \frac{R}{2} k_d^2  } $ (cf. Lemma \ref{lem:Fourierindicatori})
in the integral \eqref{eq:lastIntDDGG1} by symmetry then we use \eqref{eq:CdboundedCd-1} with $j=0$ and integrate with respect to the $d$-th coordinate, obtaining that 
 the previous expression is bounded from above by
\begin{multline*}
C_2\frac{K^2}{\sqrt{2\pi}(d-1)^2}  \int_{ [- \pi, \pi]^{d-1}  } \, \frac{\dd^{d-1} k}{(2 \pi)^{d-1}} \,\Bigg(\frac{1+\sum_{i=1}^{d-1}\cos(k_i)}{1-\frac{1}{d-1}\sum_{i=1}^{d-1}\cos(k_i)}\Bigg)^2 \\ \le 2 C_2
\frac{K^2}{\sqrt{2\pi}}\Bigg(\frac{1}{ (d-1)^2} \left\| \hat{C}_{\frac1{2(d-1)}}^{(d-1)}\right \|_2^2 +
\left\| \frac{\hat{D}^{(d-1)}}{1 - \hat{D}^{(d-1)}}    \right\|_2^2\Bigg),
\end{multline*}
where $\hat D^{(d-1)}$ and $\hat C^{(d-1)}$ denote the functions $\hat D$ and $\hat C$ in dimension $d-1$, that is, as function of $(k_1, \ldots, k_{d-1})$. 

Let us consider the second term and expand 
$ \Big[\frac{\hat{D}^{(d-1)}}{1 - \hat{D}^{(d-1)}}\Big]^2$ as a power series, 
\begin{align*}
\left\| \frac{\hat{D}^{(d-1)}}{1 - \hat{D}^{(d-1)}}    \right \|_2^2  
 &= \frac{1}{ (2 \pi)^{d-1}  }
\sum\limits_{n=1}^{\infty} 
 \int_{[-\pi, \pi]^{d-1}} \, n  \, {\big [ \hat{D}^{(d-1)}(k)  \big ]}^{n+1} \dd^{d-1}k
 \\ & \leq 
\frac{1}{ (2 \pi)^{d-1}  }
\sum\limits_{n=1}^{d-2} 
 \int_{[-\pi, \pi]^{d-1}} \, n  \, {\big [ \hat{D}^{(d-1)}(k)  \big ]}^{n+1} \dd^{d-1}k \, \, + \, \, \sum\limits_{n=d-1}^{\infty}
 n \Big (  \frac{ \pi (d-1)}{ 4 (n+1)} \Big )^{\frac{d-1}{2}},
\end{align*}
where for the last step we used (A.27) in \cite{MadraSlade93}. Then, by reasoning as in (A.30) and below in \cite{MadraSlade93}, if $d>5$, the second term 
is $O(1/d)$ as $d\to+\infty$ (in fact, it converges even exponentially). 
For the first term we observe that 
\begin{multline*}
\sum\limits_{n=1}^{d-2} 
 \int_{[-\pi, \pi]^{d-1}}  \frac{   \dd^{d-1}k}{ (2 \pi)^{d-1}  } \, n  \, {\big [ \hat{D}^{(d-1)}(k)  \big ]}^{n+1}
\\  \leq 
3
 \int_{[-\pi, \pi]^{d-1}}  \frac{   \dd^{d-1}k}{ (2 \pi)^{d-1}  }\, {\big [ \hat{D}^{(d-1)}(k)  \big ]}^{2}
 \, + \, 5
 \int_{[-\pi, \pi]^{d-1}}   \frac{   \dd^{d-1}k}{ (2 \pi)^{d-1}  }\, {\big [ \hat{D}^{(d-1)}(k)  \big ]}^{4}
 \\+ 
  (d-1)^2
 \int_{[-\pi, \pi]^{d-1}}    \frac{   \dd^{d-1}k}{ (2 \pi)^{d-1}  }\, {\big [ \hat{D}^{(d-1)}(k)  \big ]}^{6}.
 \end{multline*}
As observed below (A.31) in \cite{MadraSlade93}, this can be further estimated from above by observing that the integrals on the right-hand side are respectively the probabilities that simple random walk in $d-1$ dimension returns to
the origin after two, four, or six steps. These probabilities are respectively
of the order ${(d-1)}^{-1}$, ${(d-1)}^{-2}$, ${(d-1)}^{-3}$, therefore all of such terms are of order ${(d-1)}^{-1}$.
Analogously, we have $\left\| \hat{C}_{\frac1{2(d-1)}}^{(d-1)}\right\|_2^2\le 1+\frac{c}{d}$.
\end{proof}

We now use this estimate to prove Lemma \ref{lem:equi510}
\begin{proof}[Proof of Lemma \ref{lem:equi510}.]
To simplify the notation, we drop $\lambda$ from the notation throughout the proof.

We start with the proof of \eqref{eq:2}.
Let $x \in \mathbb{Z}^d$ be an arbitrary point, observe that by monotone convergence
$G_z  *  G_z  * \1   (x) = \lim_{R \rightarrow \infty} G_z  *  G_z  *  \IR (x)$
and fix an arbitrary $R \in (0, \infty)$.
Below we use the inequality 
\begin{align}\label{eq:useineq}
H_z(x) := G_z(x) - \delta_0(x) &\leq 2 d z D * G_z(x)
= 2 d z D * H_z(x)+ 2 d z D (x), \qquad x \in \mathbb{Z}^d.
\end{align}
The inequality arises from ignoring the interaction of the first step of the walk and  the remainder. 

For any fixed $R>0$, \eqref{eq:useineq} gives
 \begin{align*}
 \,G_z  *  G_z  * \IR (x) 
& =  ( H_z + \delta_0)   *  ( H_z + \delta_0)   * \IR (x) \\
& \le (2 d z)^2  D * D * G_z* G_z   * \IR^* (x)  \\ &\hspace{4cm}+  4 d z D * H_z   *  \IR^*(x)  +4 d z D   *  \IR^* (x) + \IR (x) .
\end{align*}
Then, by using \eqref{eq:useineq} again and the fact that $\IR (x)\le \frac1d$, we obtain that the previous equation is smaller than 
\begin{multline*}
 (2 d z)^2  D * D * G_z* G_z   * \IR^* (x) +  (4 d z)^2 D *  D * G_z   *  \IR^*(x)  +4 d z D   *  \IR^* (x)+ \frac1d \\
 \leq  5  (2 d z)^2  D * D * G_z* G_z   * \IR^* (x)   + 4 d z D   *  \IR^* (x) + \frac1d.
\end{multline*}
The first term is bounded by \eqref{eq:DDGG1}. 
For the second term we observe that
from the definition of  $\IR^* (x)$ it follows that  $D   *  \IR^* (x) \leq \frac{1}{d}$ for every $x \in \mathbb{Z}^d$.  
Combining the two bounds we obtain \eqref{eq:2} since $z \in (0, z_c)$
and since $z_c = \frac{1}{\mu} \leq \frac{1}{d-1}$ (as it follows from the comparison between the prudent walk and the directed walk) .

\smallskip
We now move to the proof of  \eqref{eq:3}. 
To begin, note that, for any $x \in \mathbb{Z}^d$,
\begin{align}
\sY^\lambda_{z,k} * \IR (x) = \sY^\lambda_{z,k} * \IR^* (x),
\end{align}
and that for any $L^2$ function $g$ 
$$
\sum_{x\in \mathbb Z^d}  \cos(k \cdot x) g(x) e^{i x \ell}
=\frac 12 \big (\hat g(k+\ell) +\hat g(k-\ell) \big ).
$$
Therefore, using  the general fact that $\|g\|_{\infty} \leq \| \hat{ g }  \|_1=\int_{[-\pi,\pi]^d}\frac{\dd k}{(2\pi)^d}|\hat g(k)|$ with $g= \sY^\lambda_{z,k}*\1^0$, we obtain,
\begin{equation}\label{eq:somestep}
\begin{split}
\|    \sY^\lambda_{z,k} *  \IR^*  \|_{\infty}  \leq 
\|    \hat{ \sY}_{z,k} \cdot \hIR^*  \|_{1}
&\le 
\left\|\Big( \hat{ G_z } - \frac{1}{2} \big (
\hat{ G_z  }(\cdot+k) + \hat{ G_z}(\cdot-k)  \big ) \Big)
\, \hIR^* \right\|_{1} \\ &\le
\frac12 \left\|   \Delta_k   (\hat G_z) \, \,    \hIR^* \right\|_{1},
\end{split}
\end{equation}
where the value of $k$ remains fixed in the $L^1$  norm in the right-hand side of the second inequality.
Using the bootstrap assumption $f_3(z)\le K$ and again Remark \ref{rem:1pos} we now obtain that
\begin{multline}\label{eq:threeterms}
\frac12 \left\|    \Delta_k   (\hat G_z) \, \,    \hIR^* \right\|_{1} \leq K \frac{1}{\hat{C}_{p(z)}(k)} 
\Bigg\{ \int_{[-\pi, \pi]^d } \frac{\dd \ell}{(2\pi)^d} \, \hat{C}_{\frac{1}{2d}}(\ell-k) \, \hat{C}_{\frac{1}{2d}}(\ell) \, \hIR^* (\ell) \, \, 
\\  + \, \, 
 \int_{[-\pi, \pi]^d } \frac{\dd \ell}{(2\pi)^d} \, \hat{C}_{\frac{1}{2d}}(\ell+k) \, \hat{C}_{\frac{1}{2d}}(\ell) \, \hIR^*  (\ell)\,\\
  + \, \, 
  \int_{[-\pi, \pi]^d } \frac{\dd \ell}{(2\pi)^d} \,  \hat{C}_{\frac{1}{2d}}(\ell-k) \,\hat{C}_{\frac{1}{2d}}(k+\ell) \, \hIR^* (\ell) \Bigg \},
  \end{multline}
where we have used the definition of $U(k,\ell)$
obtaining the sum of three terms. 
We use the following upper bound for $j \in \{-1, 0, 1\}$,
\begin{equation}\label{eq:CdboundedCd-1}
\begin{split}
\hat{C}_{\frac{1}{2d}}(\ell + j k) &  =  \frac{1}{  1 - \frac{1}{d} \sum\limits_{i=1}^{d} \cos (\ell_i + j k_i)    }
\\ & \leq  \frac{1}{  1 - \frac{1}{d} \sum\limits_{i=1}^{d-1} \cos (\ell_i + j k_i)   - \frac{1}{d}  }  \\ & =
 \frac{d}{ d -1}   \frac{1}{  1 - \frac{1}{d-1} \sum\limits_{i=1}^{d-1} \cos (\ell_i + j k_i)    } =: 
 \frac{d}{ d -1}   \hat{C}^{ (d-1)  }_{\frac{1}{2(d-1)}}( \tilde  \ell + j \tilde k),
\end{split}
\end{equation}
 where $\tilde \ell$ is the restriction of $\ell$ to the first $d-1$ coordinates. 
Therefore, we first replace  $\hIR^* (\ell)$ by 
 $ C_2\sqrt{R}  e^{ -  \frac{R}{2}\ell_d^2  } $ (cf. Lemma \ref{lem:Fourierindicatori})
in the integrals  (\ref{eq:threeterms}) by symmetry, then we use \eqref{eq:CdboundedCd-1} and  integrate with respect to the $d$-th coordinate, obtaining that  (\ref{eq:threeterms}) is bounded from above by $$C_2\sqrt{\frac{2}{ \pi }} \,  \Big  ( \frac{d}{d-1} \Big )^3 \frac{1}{ \hat{C}_{p(z)}(k)  } ({J}_{1,-1}(\tilde k)+{J}_{0,-1}(\tilde k)+{J}_{0,1}(\tilde k)),$$ 
where 
$$
 {J}_{j, j^\prime}(\tilde k) = \int_{ [-\pi, \pi]^{d-1}  } \frac{  \dd^{d-1} \ell }{ (2 \pi)^{d-1}} \frac{1}{\big ( 1  - \hat{D}(\ell + j\tilde k) \big )   \big ( 1  - \hat{D}(\ell + j^\prime \tilde k) \big ) }.
$$
By the Cauchy-Schwarz inequality we have that $ {J}_{j, j^\prime}(\tilde k)\le  {J}_{0,0}$. Note that ${J}_{0,0}$ is independent of $\tilde k$ and if $d>5$ it is bounded by $1+O(1/d)$, see (5.2) and (5.10) in \cite{SladeBook}.
\end{proof}
 
The next lemma is a consequence of our diagrammatic bounds and of the previous lemma:

\begin{lemma}\label{lem:equi511}
Fix $z \in (0, z_c)$.
There exists  
 $\lambda_0\in (0,1)$ sufficiently small, $d_0\in \mathbb N$ sufficiently large and a positive and finite constant $\tilde{c}$ independent from $z$, $\lambda$, and $d$, such that for any $K \in (0, 4]$, if $\lambda = 1$ and $d > d_0$ or  if $d > 5$ and $\lambda < \lambda_0$ we then have,
 \begin{equation}\label{eq:boundPiNSUmZ}
 \sum\limits_{z \in \mathbb{Z}^d} | \Pi^\lambda_z(x)| \leq  \tilde{c} \, \,  \frac{\lambda}{d},
 \end{equation}
  \begin{equation}\label{eq:boundPiNSUmZcos}
 \sum\limits_{z \in \mathbb{Z}^d}
 \big[1 - \cos (k \cdot x)\big]
 | \Pi^\lambda_z(x)| \leq   \tilde{c} \,  \,  \frac{1}{\hat C_{p(z)}(k)} \, \, \frac{1}{ (1 - \frac{\lambda K^2}{d})^3  }
 \end{equation}
 \end{lemma}
 \begin{proof}
From an immediate application of Theorem \ref{thm:equi41} and of  Lemma \ref{lem:equi510} we obtain that, for any integer $N \geq 1$,
$$
\sum\limits_{x \in \mathbb{Z}^d}  \Pi_z^{(N)}(x) \leq 
  \, (2 d z \lambda )^N \, (\sB_z^\lambda)^{N} \, \leq \,
\Big ( \frac{2 d z \lambda  K^2 \underline{c}}{d} \Big )^{N}.
$$
 Now by observing that, 
 \begin{equation}\label{eq:relationzd}
 2 d z \leq \frac{2d}{d-1} \leq 3,
 \end{equation}
 for each $z \in [0, z_c)$ and $d > 5$ and by summing over all $N \geq 1$ we deduce 
  \eqref{eq:boundPiNSUmZ}.
We now prove \eqref{eq:boundPiNSUmZcos}. Note that we obtain from Theorem \ref{thm:equi41}, 
 Lemma \ref{lem:equi510} 
 and from the fact that for any $z \in [0, z_c)$,
 $\hat{C}_{p(z)}(k) \leq \hat{C}_{\frac{1}{2d}}(k) = \frac{1}{1 - \hat{D}(k)}$, that, for  any integer $N \geq 1$,
 \begin{align}
 \notag
 \sum\limits_{x \in \mathbb{Z}^d} [ 1 - \cos(k \cdot x) ] \Pi_z^{(N)}(x)  & \leq 6 N^2 \lambda^N  (dz)^N
 \Big ( (\sB_z^\lambda)^{N-1} \|  \sY^\lambda_{z,k} * \1  \|_\infty
 + (\sB_z^\lambda)^{N} (1 - \hat{D}(k) \Big) 
 \notag  \\
 & \leq  \notag  \frac{1}{\hat{C}_{p(z)}}
 18 N^2 (dz)^N \, \, \lambda \, \, \Big  (\frac{ \lambda \,  K^2 \underline{c}}{d} \Big)^{N-1} 
 \Big ( K + \frac{K^2}{d}       \Big )
 \end{align}
 Now, by summing over all $N \geq 1$,
  using again \eqref{eq:relationzd} 
 and our assumptions on $\lambda$, $d$ and $K$,
  we obtain \eqref{eq:boundPiNSUmZcos} as desired. 
\end{proof}

\subsection{Bootstrap argument completed}
We now complete the bootstrap argument which will allow us to prove Theorem  \ref{thm:BubbleCond}. Contrary to the previous sections, this  part of the argument is quite standard and not  model dependent once 
 one is in possession of Lemma \ref{lem:equi511}.
 Hence we only present the main steps and refer to  \cite[Lemma 5.16]{SladeBook}.
\begin{lemma}\label{lem:neededforbootstrap}
There exists $d_0 \in \mathbb{N}$ large enough and $\lambda_0 \in (0, 1)$ small enough such that the following holds. 
Fix $z \in (0, z_c)$ and suppose that $f(z) \leq 4$.
Then, if $d \geq d_0$ and $\lambda \in (0, \lambda_0)$ 
or if $\lambda=1$ and $d > d_0$,
it is in fact the case that 
$f(z) \leq  1 +  c^* \frac{\lambda}{d} < 4$,
for some $c^*$ which is independent of $z$, $\lambda$ and $d$.
\end{lemma}
\begin{proof}
The proof of this lemma is identical to the proof of 
 \cite[Lemma 5.16]{SladeBook}. The  
only difference is that
the term $\beta$ in  \cite[Lemma 5.16]{SladeBook}
has to be replaced by 
$\frac{\lambda}{d}$ everywhere. 
We sketch the main steps. 
In the proof  $c$ is some positive constant  which does not depend on $\lambda$, $z$ and $d$ and which may differ from  line to line. 
For $f_1(z)$ note that $\chi(z) > 0$ (cf. \eqref{eq:susceptibility}) so that, by using (\ref{eq:decompositionGx}) and $\hat{G}(0) = \chi(z)$,
$$
\chi(z)^{-1} = 1 - 2 d z - \hat{\Pi}_z(0) > 0.
$$
Therefore, by Lemma \ref{lem:equi511},
we deduce that, under the assumptions of the lemma,
$$
f_1(z) = 2d z < 1 - \hat{\Pi}_z(0) \leq 1 + \tilde{c} \,  \frac{\lambda}{d},
$$
as desired.

For $f_2(z)$, we first define $\hat{F}_z(k) : = 1 / \hat{G}_z^\lambda(k)$ so that 
\begin{equation}\label{eq:boundforf2}
\frac{\hat{G}^\lambda_z(k)}{\hat{C}_{p(z)}(k)} =
\frac{1 - 2 \, d \,  p(z) \,  \hat{D}(k)   }{ \hat{F}^\lambda_z(k)}
=1 + \frac{ 1 - 2 \, d \, p(z) - \hat{F}^\lambda_z(k)  }{\hat{F}^\lambda_z(k)}.
\end{equation}
The goal is to show that the second term in the right-hand side of (\ref{eq:boundforf2})
is bounded from above by a constant times $\frac{\lambda}{d}$.
The first step is to obtain an upper bound on the  numerator such term.  Following the same steps as \cite[(5.54)-(5.57)]{SladeBook}
and applying our Lemma \ref{lem:equi511},
we obtain that the numerator of the second term in (\ref{eq:boundforf2}) is bounded from above by a constant times,
$$
\frac{\lambda}{d} \, \big [  \hat{F}^\lambda_z(0) + \big (1 - \hat{D}(k) \big )  \big ].
$$
 Following the same steps as \cite[(5.58)-(5.60)]{SladeBook}
 and using again  Lemma \ref{lem:equi511},
 we obtain that the denominator of the second term in the right-hand side of 
 (\ref{eq:boundforf2}) is bounded from below by,
 $$ 
 \Big ( \frac{1}{2} - c \frac{\lambda}{d} \big )  \Big [ \hat{F}^\lambda_z(0) + (1 - \hat{D}(k)) \Big ],
 $$
 for some constant $c>0$ which does not depend on $\lambda$, $z$ and $d$.
 The combination of the two bounds gives that,
 $
 f_2(z) \leq  1 + c \frac{\lambda}{d}.
 $

 We now move to $f_3(z)$.
 We define $\hat{g}_z(k) =   2 d z  \hat{D}(k)  + \hat{\Pi}_z(k)$ such that 
 $$
 \hat{G}_z^\lambda(k) = \frac{1}{1 - \hat{g}_z(k)}.
 $$
 Since also in our case $g_z(x) = g_z(-x)$,  we can follow the same steps as 
 \cite[(5.61)-(5.64)]{SladeBook} and obtain, using our upper bounds on $f_1(z)$ and $f_2(z)$, that $f_3(z) \leq c \,  \frac{\lambda}{d}$, thus concluding the proof. 
\end{proof}
We are now ready to prove that the critical prudent bubble diagram is finite under appropriate assumptions on $d$ and $\lambda$.
\begin{proof}[Proof of Theorem  \ref{thm:BubbleCond}]
Let $d_0 \in \mathbb{N}$ and $\lambda_0 \in (0,1)$ be as in Lemma \ref{lem:neededforbootstrap}, suppose that $\lambda = 1$ 
and that $d > d_0$ or that $d > 5$ and $\lambda \in (0, \lambda_0)$.
First, note that  $f$ is continuous in $[0, z_c]$ and that $f(0) =1$ (the proof of these claims is simple and is completely analogous to that of Lemma 5.13 and 5.14 in \cite{SladeBook}).
Choose then $z$ small enough so that  $f(z) \leq 4$.
Then,  from Lemma \ref{lem:neededforbootstrap},
from the continuity of $f$ and from the simple lemma \cite[Lemma 5.9]{SladeBook} 
we deduce that,
$$
\forall z \in [0, z_c)  \quad \quad f(z) \leq 1 + c^* \frac{\lambda}{d} < 4.
$$
From this bound and by using  \eqref{eq:2} with $K =1 + c^* \frac{\lambda}{d} < 4$
we then deduce that for each $x \in \mathbb{Z}^d$, 
\begin{equation*}
 \forall z \in [0, z_c)  \quad \quad 
G^\lambda_z* G^\lambda_z* \1(x) \leq  \sB^\lambda_z \leq (1 + c^* \frac{\lambda}{d})^2 \, \,  \frac{ \underline{c} }{d} \leq   \frac{16 \,  \underline{c} }{d}.
 \end{equation*}
 Hence,  by the Abel's limit theorem,  for each $x \in \mathbb{Z}^d$, 
$$
G^\lambda_{z_c}* G^\lambda_{z_c}* \1(x) 
= \lim\limits_{z \rightarrow z_c} 
G^\lambda_{z}* G^\lambda_{z}* \1(x)  
\leq \frac{16 \, \underline{c}}{d}.
$$
Taking the sup over all $x \in \mathbb{Z}^d$ we thus obtain that
$$
 \sB^\lambda_{z_c} \leq  \frac{16 \, \underline{c}}{d}.
$$
This concludes the proof.
\end{proof}

We now state some further consequences of the completed bootstrap argument, which we need in the sequel. 

\begin{corollary}
Under the conditions of Theorem \ref{thm:BubbleCond}, for $k\in[-\pi,\pi)^d$, 
\begin{eqnarray}
	\| \sY^\lambda_{z_c,k} *  \1 \|_{\infty}
	&\le& O\big(|k|^2\big),\label{eq:GdisplBd}\\
	\| \sY^\lambda_{z_c,k} *  G^\lambda_{z_c} * \1 \|_{\infty}
	&\le& O\big(|k|^{(d-5)\wedge2}\big).\label{eq:GGdisplBd}
\end{eqnarray}
\end{corollary} 
 \proof
Lemma \ref{lem:neededforbootstrap} gives $f(z)\le4$ for $z<z_c$, and hence 
\eqref{eq:3} is valid for all $z<z_c$ with $K=4$. Monotone convergence then implies
\[\| \sY^\lambda_{z_c,k} *  \1 \|_{\infty}\leq    \frac{3 \,  \underline{c} \, 4 \,  }{\hat C_{1}(k)}
\leq c|k|^2,\] 
which is \eqref{eq:GdisplBd}. 

For \eqref{eq:GGdisplBd} we need a variation upon \eqref{eq:3}. 
Following the same steps as 
for  (\ref{eq:somestep}) and  (\ref{eq:threeterms}) (the only difference is the presence 
of an additional term $|\hat{G}(\ell)|$ in the integrand) we obtain
that 
\begin{multline}\label{eq:beforechen}
	\| \sY^\lambda_{z,k} *  G^\lambda_{z} *\ \1 \|_{\infty} 
= \lim\limits_{R \rightarrow \infty} 
	\| \sY^\lambda_{z,k} *  G^\lambda_{z} * \IR \|_{\infty} 
 \\ \leq 4 \, 
  \lim\limits_{R \rightarrow \infty} 
  \frac{1}{\hat{C}_{p(z)}(k)} 
\Big ( \int_{[-\pi, \pi]^d } \frac{\dd \ell}{(2\pi)^d} \,  \hat{C}^2_{\frac{1}{2d}}(\ell)   \hat{C}_{\frac{1}{2d}}(\ell-k) \, \hIR^* (\ell) \, \, 
\\  + \, \, 
 \int_{[-\pi, \pi]^d } \frac{\dd \ell}{(2\pi)^d} \  \hat{C}^2_{\frac{1}{2d}}(\ell) \, \hat{C}_{\frac{1}{2d}}(\ell+k) \, \, \hIR^*  (\ell)\,\\
  + \, \, 
  \int_{[-\pi, \pi]^d } \frac{\dd \ell}{(2\pi)^d} \,    \hat{C}_{\frac{1}{2d}}(\ell) \,  \hat{C}_{\frac{1}{2d}}(\ell-k) \,\hat{C}_{\frac{1}{2d}}(k+\ell) \, \hIR^* (\ell) \Big ).
  \end{multline}
We first replace $\hIR^* (\ell)$ by $d \sqrt{R \pi}  e^{ -  R \ell_d^2  } $ 
in the integrals  (\ref{eq:beforechen}) by symmetry and then we integrate with respect to the $d$-th coordinate. Then, by using \eqref{eq:CdboundedCd-1} we obtain that  (\ref{eq:beforechen})  is bounded from above by $\frac{4 d}{2 \pi } \,  \big  ( \frac{d}{d-1} \big )^3 \frac{1}{ \hat{C}_{p(z)}(k)  } (\hat{J}_{1,-1}(\tilde k)+\hat{J}_{0,-1}(\tilde k)+\hat{J}_{0,1}(\tilde k))$ where 
$$
 \hat{J}_{j, j^\prime}(\tilde k)  : = \int_{ [-\pi, \pi]^{d-1}  } \frac{  \dd^{d-1} \ell }{ (2 \pi)^{d-1}} \frac{1}{  \big ( 1 - \hat{D}(\ell)  \big ) \big ( 1  - \hat{D}(\ell + j\tilde k) \big )   \big ( 1  - \hat{D}(\ell + j^\prime \tilde k) \big ) }.
 $$
 By the Cauchy-Schwarz inequality and the periodicity of the integrand we deduce  that 
  $ \hat{J}_{-1,  1}(\tilde k) \leq    \hat{J}_{0,  1}(\tilde k) =   \hat{J}_{0,  -1}(\tilde k) . $
  (see \cite[(4.28)]{ChenSakai}).
  Therefore, by monotone convergence and by the bound (4.30) in \cite{ChenSakai} we get that
  \begin{equation}\label{eq:afterchen}
  \| \sY^\lambda_{z_c,k} *  G^\lambda_{z_c} *\ \1 \|_{\infty} \le 
  \frac{12 d}{2 \pi } \,  \big  ( \frac{d}{d-1} \big )^3 \frac{1}{ \hat{C}_{\frac{1}{2d}}(k)  } \hat{J}_{0,  1}(\tilde k) \le C_d |k|^2\, |\tilde k|^{(d-7)\wedge 0}=C_d |k|^{(d-5)\wedge 2}.
  \end{equation}
  In the last inequality we used that $|\tilde k| \le |k|$, where in the previous display $|\tilde k|^a$ denotes the $d-1$ dimensional Euclidean norm of $\tilde k \in \mathbb{R}^{d-1}$ to the power of $a$,  while $|k|^2$ is the $d$-dimensional Euclidean norm of $k \in \mathbb{R}^d$.
  This concludes the proof.


 \qed

 \section{The critical two-point function}\label{sec:CritG}
In this section, we use the results from the previous sections to prove the asymptotics for the critical two-point functions as formulated in Theorem \ref{thm:CritG}. 
We follow a well-established strategy that has proved useful in similar cases. 
Let us define for any $n, N \in \mathbb{N}_0$ and $x \in \mathbb{Z}^d$,
\begin{equation}\label{eq:pinNdefinition}
\pi^{(N)}_n(x) = \lambda^N
\sum\limits_{   \omega \in \cW_n(0,x) } \sum\limits_{ L \in \mathcal{L}_N[0,n]}
\prod_{ s t \in L}  (-\mathcal{U}_{st})  \prod_{  s^\prime t^\prime \in \mathcal{C}(L)} (1 + \lambda  \, \mathcal{U}_{s^\prime t^\prime}),
\end{equation}
\begin{equation}\label{eq:pindefinition}
\pi_n(x) = \sum\limits_{N=1}^{\infty} (-1)^N \pi^{(N)}_n(x),
\end{equation}
so that 
$$
\Pi_z(x) = \sum\limits_{n=1}^{\infty} \pi_n(x)  z^n.
$$
In order to study the asymptotic behavior of $c_n^\lambda(x)$, we investigate the divergence of its generating function $	G^\lambda_z(x)=\sum_{n \geq 0 }\ z^n c^\lambda_n(x)$ near its radius of convergence $z_c$, and then apply a Tauberian theorem to deduce the asymptotics of $c^\lambda_n(x)$ as $n\to\infty$. 
It is for this purpose that we allow $z$ to be complex-valued throughout this section. 

The key are strong moment estimates on the lace expansion coefficients $\pi_n$, which we formulate in the next lemma. 
For brevity, we remove $\lambda$ from the notation throughout the section. 

\begin{lemma}[Spatial and temporal moments of lace expansion coefficients]\label{lem:moments}
Under the conditions of Theorem \ref{thm:BubbleCond}, for $\varepsilon\in(0, \frac{d-5}{2})$ and $\delta\in(0,2\wedge(d-5))$, 
\begin{equation}
	\sum_x\sum_{n\ge2}n(n-1)^\varepsilon|\pi_n(x)|\,z_c^{n-1}<\infty \label{eq:TempDerBd}
\end{equation}
and 
\begin{equation}
	\sum_x\sum_{n\ge2}|x|^{2+\delta}|\pi_n(x)|\,z_c^{n-1}<\infty. \label{eq:SpaceDerBd}
\end{equation}
\end{lemma}
\proof
We sketch the argument for \eqref{eq:TempDerBd}. The proof of Corol.\ 6.4.3 in \cite{MadraSlade93} gives that the left hand side is bounded from above by a constant times
\[ \sup_y 
\sum\limits_{x \in \mathbb{Z}^d} 
\sum_{n\ge2} \,n \, (n-1)^\varepsilon c_n(x) \,   \I{y}^0(x) z_c^{n-1} = \sup_{ y \in \mathbb{Z}^d }  \, \, \delta_z^\epsilon \, \,  \partial_z \,  \,   G_{z_c} * \1^0 (y),
\]
where $\delta_z^\epsilon$ is the fractional derivative,
as defined in (6.3.1) in \cite{MadraSlade93},
and $\partial_z$ is the partial derivative with respect to $z$.
By following the same steps as in the proof of Lemma  6.3.1 in
\cite{MadraSlade93}
(the only difference is the presence of Fourier transform of the function  $\1^0$
in the integral (6.4.12) in
\cite{MadraSlade93}, which is positive), we obtain 
that the previous term  is bounded from above by a constant times
\begin{equation}\label{eq:DBd1}
\int\sum_{n\ge2} n(n-1)^\varepsilon \hat D(k)^{n-2} \hat{\1^0} (k) \dk.
\end{equation}
Lemma \ref{lem:boundDind} gives that the right hand side of \eqref{eq:DBd1} is bounded from above by  
\begin{equation}
C \sum_{n\ge2}n(n-1)^\varepsilon n^{-\frac{d-1}{2}},
\end{equation} 
with $C$ a positive constant. This is finite if $1+\varepsilon-\frac{d-1}{2}<-1$.

We now move to the proof of (\ref{eq:SpaceDerBd}).
Using the integral representation as in (2.30) in
\cite{Heyde11}, Lemma 4.1 in
\cite{ChenSakai} and the rotational symmetry of the functions $\pi_n^{(N)}$ we obtain, as in (2.30) in
\cite{Heyde11}, that the left-hand side of 
(\ref{eq:SpaceDerBd})
is bounded from above by
a constant times
\begin{equation}\label{eq:integraldudv}
\int_{0}^{\infty}  \frac{\dd u }{u^{1 + \delta_1}} 
\int_{0}^{\infty}  \frac{\dd v }{v^{1 + \delta_2}} 
\sum\limits_{x \in \mathbb{Z}^d} 
\sum\limits_{N=2}^{\infty}
\big (
1 - \cos( u x_1) 
\big )
\big ( 
1 - \cos( v x_1) 
\big )
\pi_{n}^{N}(x) z_c^n,
\end{equation}
where
$$
\delta_1 \in (\delta,  2 \wedge {d-5}),  \quad  \delta_2 = 2+ \delta  - \delta_1.
$$
The same as in the proof of \cite[Lemma 2.4]{Heyde11}, we divide the integral above into four integrals,  $I = I_1 + I_2 + I_3 + I_4$, 
where $I_1$, $I_2$, $I_3$, $I_4$ are defined as  (\ref{eq:integraldudv}) but with integration ranges  respectively $\int_{0}^1 \dd u  \ldots \int_{1}^{\infty} \dd v \ldots$
for $I_1$, 
$\int_{0}^1 \dd u  \ldots \int_{1}^{\infty} \dd  v  \ldots$ for $I_2$,
$\int_{1}^\infty \dd u  \ldots \int_{0}^{1} \dd v \ldots $ for $I_3$,
and
$\int_{1}^\infty \dd u  \ldots \int_{1}^{\infty} \dd v  \ldots$
for $I_4$.
The bound $I_4 < \infty$ follows from $1 - \cos(t) \leq 2$
and (\ref{eq:TempDerBd}).
To upper bound $I_1$, $I_2$ and $I_3$ we 
use a slight modification of the arguments which allowed the derivation of (\ref{eq:bounds2b}),
obtaining that, 
\begin{multline}
\sum\limits_{x \in \mathbb{Z}^d} 
\sum\limits_{N=2}^{\infty}
\, 
\big (
1 - \cos( u \,  x_1) 
\big )
\, 
\big ( 
1 - \cos( v \,  x_1) 
\big )
\, \pi_{n}^{(N)}(x) \, z_c^n \\ 
\leq  
C_1 N^4 \,  \lambda^N \,  \max \{  (dz)^N, 1  \}   
(B^\lambda_z)^{N-2}
\Big ( 
\, 
  \|   \sY^\lambda_{z,v} * \1    \|_\infty \, 
  \|      \sY^\lambda_{z,u}  * G_z^\lambda  \, * \1  \,   \|_\infty  
  \Big ) \\ + 
  C_2 N^3  \, ( d z \lambda )^N \, 
(B^\lambda_z)^{N-1}
\big (1 - \hat{D}(v) \big )
  \|      \sY^\lambda_{z,u}  * G_z^\lambda  \, * \1  \,   \|_\infty  
  \le C N^4 \Big(c \frac{\lambda}{d}\Big)^N v^2 u^{(d-5)\wedge 2},
\end{multline} 
where  $c$ is independent of the dimension.  By replacing the upper bound above in the integrals (\ref{eq:integraldudv}) we deduce  by our choice of $\delta_1$ and $\delta_2$,  that such integrals converge,  hence  concluding the proof.
\qed 

\begin{lemma}\label{lem:PiNullPiK}
Under the conditions of Theorem \ref{thm:BubbleCond}, for $z\le z_c$ there exists $K_z>0$ such that 
\[ \hat\Pi_z(0)-\hat\Pi_z(k)=\big(K_z+o(1)\big)\, \big(1-\hat D(k)\big)\]
as $|k|\to0$, with $K_z= \sum\limits_{x \in \mathbb{Z}^d} \sum\limits_{n \geq 1} |x|^2 \pi_n(x) z^n$.
\end{lemma}
Mind that Lemma \ref{lem:moments} guarantees that $K_z$ is finite. 
\proof The left hand side may be rewritten as 
\[ \sum_x\sum_{n\ge1}(1-\cos(k\cdot x))\pi_n(x)z^n\]
Using $1-\cos(k\cdot x)=\frac12(k\cdot x)^2+O(k\cdot x)^{2+\delta}$ with $\delta$ as in Lemma \ref{lem:moments} and the spatial symmetry of $\pi_n$ we deduce by Lemma   \ref{lem:moments} that 
\begin{align}
\sum_x\sum_{n\ge1}(1-\cos(k\cdot x)\pi_n(x)z ^n 
&= \sum_x\sum_{n\ge1}\sum_{i,j=1}^d\frac12k_ik_jx_ix_j\pi_n(x)z ^n + O(|k|^{2+\delta})\sum_x\sum_{n\ge1}|x|^{2+\delta}\pi_n(x)z ^n \nonumber \\
&= \frac{ |k|^2}{2d}\sum_x\sum_{n\ge1}|x|^2\pi_n(x)z^n + o( |k|^2).
\end{align}
Noting that $1-\hat D(k)=\big( \frac{1}{2d} \, + \, o(1)\big)\,|k|^2$ (cf.\ \eqref{eq:Dk})  and comparing 
this expansion with the previous identity concludes the proof. 
\qed

\subsection{Proof of Theorem \ref{thm:CritG}}
We now start the proof of Theorem \ref{thm:CritG}, and we follow the strategy in \cite[Section 2]{Heyde11}. The starting point is the expansion identity \eqref{eq:decompositionGx} (in Fourier form), which we rewrite for $|z|<z_c$ as 
\begin{equation}
  z_c\,\hat G_z(k)
  = \frac{1}{[1-z/z_c]\left(A(k)+E_z(k)\right)+B(k)} 
  \label{eqProof7}
  = \frac{1}{[1-z/z_c]\,A(k)+B(k)} -\Theta_z(k),
\end{equation}
with 
\begin{eqnarray}
  \label{eqProof3}
  A(k) &=&  \hat D(k)+\partial_z\hat\Pi_z(k)\big|_{z=z_c},\\
  \label{eqProof4}
  B(k) &=&  1-\hat D(k)+\frac1{z_c}\left(\hat\Pi_{z_c}(0)-\hat\Pi_{z_c}(k)\right),\\
  \label{eqProof5}
  E_z(k) &=& \frac{\hat\Pi_{z_c}(k)-\hat\Pi_{z}(k)}{z_c-z}-\partial_z\hat\Pi_z(k)\big|_{z=z_c},
\end{eqnarray}
and
\begin{equation}\label{eqProof8}
    \Theta_z(k)=\frac{[1-z/z_c]\,E_z(k)}
                    {\big([1-z/z_c]\left(A(k)+E_z(k)\right)+B(k)\big)\,
                     \big([1-z/z_c]\,A(k)+B(k)\big)}.
\end{equation}
For the first term in (\ref{eqProof7}) we write
\begin{equation}\label{eqProof9}
    \frac{1}{[1-z/z_c]\,A(k)+B(k)}
    =\frac{1}{A(k)+B(k)}\sum_{n=0}^\infty\left(\frac
     z{z_c}\right)^n\left(\frac{A(k)}{A(k)+B(k)}\right)^n,
\end{equation}
and the geometric sum converges whenever $|z|<z_c\left(A(k)+B(k)\right)/A(k)$; the latter term approximates $z_c$ as $|k|\to0$.
For $|z|<z_c$, we can expand $\Theta_z(k)$ into a power series for $|z|<z_c$ as
\begin{equation}\label{eqProof10}
    \Theta_z(k)=\sum_{n=0}^\infty\theta_n(k)\,z^n,
\end{equation}
and thus get 
\begin{equation}\label{eqProof11}
    \hat c_n(k)=\frac1{z_c}\left(\frac{z_c^{-n}}{A(k)+B(k)}\left(\frac{A(k)}{A(k)+B(k)}\right)^n-\theta_n(k)\right),
    \quad
    \hat c_n(0)=\frac1{z_c}\left(\frac{z_c^{-n}}{A(0)}-\theta_n(0)\right).
\end{equation}

\begin{lemma}\label{lem:GzNull} 
Under the conditions of Theorem \ref{thm:BubbleCond} there exist  constants $c_1, c_2 > 0$ such that
for every $k \in [0, \pi)^{d}$, 
\begin{equation}\label{eq:suscbound1}
z^{-1}_c \, \hat{G}(k)^{-1} \geq z^{-1}_c  \, \hat{G}(0)^{-1}  \geq  c_1 |z_c - z|
\end{equation}
and 
\begin{equation}\label{eq:suscbound2}
z_c^{-1} \hat{G}(0)^{-1}  \leq  c_2 \,  |z_c - z|.
\end{equation}
\end{lemma}
\proof 
The first inequality in (\ref{eq:suscbound1}) follows from the definition of the Fourier transform
by taking the absolute value inside the sum over $x$. 
The second inequality in  (\ref{eq:suscbound1}) and    (\ref{eq:suscbound2})
follow from   \eqref{eqProof7}
since   $A(0)>0$, $B(0)=0$ and by Lemma \ref{lem:moments} we have that
 $E_z(k) \leq O (|z-z_c|^{\epsilon})$ for any $\epsilon$ as in 
 Lemma \ref{lem:moments}
 (for the proof see  
 \cite[(2.64)]{Heyde11}).
 \qed

\begin{lemma}\label{lem:thetaBd}
Under the conditions of Theorem \ref{thm:BubbleCond}, 
$|\theta_n(k)|\le O(z_c^{-n}n^{-\varepsilon})$ for all $\epsilon \in  (0,  \frac{d-5}{2})$  uniformly in $k$. 
\end{lemma}
\begin{proof}
Given Lemmas \ref{lem:moments} and \ref{lem:GzNull}, the proof is model independent and follows line by line the proof of   \cite[Lemma  2.1]{Heyde11}. 
Mind that it is in this proof where we need $z$ to be complex-valued.
\end{proof}
As a consequence of Lemma \ref{lem:thetaBd} and \eqref{eqProof11} we get the following result:
\begin{corollary}
\label{cor:thetancn}
Under the conditions of Theorem \ref{thm:BubbleCond}, for any $\epsilon \in ((d-5)\wedge 2)$ 
$$
\hat c_n(0)= \big ( z_c \,  A(0) \big )^{-1} z_c^{-n} \big (1+O(n^{-\epsilon}) \big ).
$$
\end{corollary}
\proof[Proof of Theorem \ref{thm:CritG}.]
Let $k\in\R^d$, and assume $n$ is large enough so that $k/\sqrt n \in [-\pi, \pi)^{d}$. Choose $\epsilon \in  (0,  \frac{d-5}{2})$, then 
\eqref{eqProof11} and Lemma \ref{lem:thetaBd} imply that 
\begin{eqnarray}\label{eq:Proof122}
    \frac{\hat c_n(k/\sqrt n)}{\hat c_n(0)}
    &=&\big(1+O(n^{-\varepsilon})\big)\frac{A(0)}{A(k/\sqrt n)+B(k/\sqrt n)}\left(\frac{A(k/\sqrt n)}{A(k/\sqrt n)+B(k/\sqrt n)}\right)^n+O(n^{-\varepsilon}).
\end{eqnarray}
Expand 
\begin{equation}\label{eqProof12}
	\frac{A(k/\sqrt n)}{A(k/\sqrt n)+B(k/\sqrt n)}
	=1-\frac{n(1-\hat D(k/\sqrt n))\,\big(A(k/\sqrt n)+B(k/\sqrt n)\big)^{-1}\,{B(k/\sqrt n)}(1-\hat D(k/\sqrt n))^{-1}}{n},
\end{equation} 
and note that, when $n\to\infty$.
\begin{eqnarray*}
	n(1-\hat D(k/\sqrt n)&\longrightarrow& \frac{1}{2d} \,|k|^2,\\
	A(k/\sqrt n)+B(k/\sqrt n)&\longrightarrow& A(0)=1+\sum_x\sum_{m\ge2}m\pi_m(x)z_c^{m-1},\\
	\frac{B(k/\sqrt n)}{1-\hat D(k/\sqrt n)}&\longrightarrow& 1+K_{z_c}/z_c 
\end{eqnarray*} 
(where we use \eqref{eq:Dk} for the first line, the definition of $A(k)$ and $B(k)$ for the second line and Lemma \ref{lem:PiNullPiK} for the last line),
which gives that ${\hat c_n(k/\sqrt n)}/{\hat c_n(0)}\to\exp\{-K|k|^2\}$ with 
\begin{equation}\label{eq:DefK}
	K^{-1}=  \frac{1}{2d} \left(1+\sum_x\sum_{m\ge2}m\pi_m(x)z_c^{m-1}\right)\left(1+K_{z_c}/z_c \right).
\end{equation} 
This concludes the proof of Theorem \ref{thm:CritG}.
\qed

\section{Convergence to Brownian motion} \label{sec:convergence}
In this section, we finally prove convergence of rescaled prudent walk to Brownian motion, i.e., we prove Theorem \ref{thm:ConvBM}. The proof consists of two parts: (i) convergence of finite-dimensional distributions and (ii) a tightness argument based on the moment estimates in Proposition \ref{prop-moments}. 
Given our previous sections,  the analysis here is standard and almost model independent,  hence we only sketch the main arguments. 
In order to prove Theorem \ref{thm:ConvBM}, it is sufficient to consider a \emph{separating class} of functions $f$. A convenient choice of separating functions is provided by the family
\[ g(x_1,\dots,x_N)=\exp\{ik\cdot(x_1,x_2-x_1,\dots,x_N-x_{N-1})\},\qquad n\in\N, x_1,\dots,x_N\in\Z^d,k\in(\R^d)^N.\]
For $\boldsymbol{n} = (n^{(1)}, \ldots, n^{(N)}) \in \mathbb{N}^{N}$, 
with $n^{(1)} < \ldots < n^{(N)}$ and $\boldsymbol{k}=(k^{(1)},\dots,k^{(N)})\in(\R^d)^N$
we define
\begin{equation}\label{defCN}
\hat{\boldsymbol{c}}_{ \boldsymbol{n}}( \boldsymbol{k}) = 
\sum\limits_{ w \in \mathcal{W}_{  n^N }   } \varphi^\lambda(w)  \exp \Big \{  i \sum\limits_{j=1}^{N} k^{(j)} \cdot  \big (w(n^{(j)}) - w(n^{(j-1)}  ) \big ) \Big \}.
\end{equation}
We also introduce the notation,
$$
\boldsymbol{k} \cdot \Delta w(\boldsymbol{n}) = \sum\limits_{j=1}^{N} k^{(j)} \cdot  \big (w(n^{(j)}) - w(n^{(j-1)}  )  \big ) .
$$
We now state our first lemma about convergence of finite dimensional distributions.
\begin{lemma}[Finite-dimensional distributions]  \label{lem:fdd}
Let $N\in\N$ and $\boldsymbol{k}=(k^{(1)},\dots,k^{(N)})\in(\R^d)^N$, let $0=t^{(0)}<t^{(1)}<\dots<t^{(N)}\le1$, and $g=(g_n)$ a sequence of real numbers satisfying $0\le g_n\le n^{-1/2}$,  define
$$
 \boldsymbol{T} = (  t^{(1)},  t^{(2)}, \ldots, t^{(N-1)},  T ), \quad \quad  \boldsymbol{k}_n =   \frac{\boldsymbol{k}}{\sqrt{n}},
$$
where $T = t^{(N)} (1 - g_n)$.
Under the conditions of Theorem \ref{thm:ConvBM} we have,
\begin{equation}\label{eqFinDimConv}
	\cnT{N}_{n\bT}\big( \boldsymbol{k}/\sqrt{n}\big)=\hat c_{n T}\big(0\big)\exp\left\{-K\,\sum_{j=1}^N|k^j|^2(t^{(j)} - t^{(j-1)} )\right\} \qquad\text{ as $n\to\infty$,}
\end{equation}
holds uniformly in $g$, 
\end{lemma}

Our next lemma shows tightness for the sequence  $\{X_n\}$:
\begin{lemma}[Tightness] \label{lem:tightness}
The sequence $\{X_n\}$ in (\ref{eq:defXn}) is tight in $D([0,1],\R^d)$. 
\end{lemma}

These lemmas enable us to prove our main theorem.  
\begin{proof}[Proof of Theorem \ref{thm:ConvBM}]
The asserted convergence in distribution is implied by convergence of finite-dimensional distributions and tightness of the sequence $X_n$, see e.g.\ \cite[Theorem 15.1]{Billi68}. Hence, the result follows from Lemma \ref{lem:fdd}, the observation that the family $g$ is separating, and the tightness in Lemma \ref{lem:tightness}. 
\end{proof}

The proof of these lemmas follows closely that of \cite[Thm.\ 1.6]{Heyde11} and \cite[Corollary 5.2]{Heyde11}. We only explain the main modifications here. 

\begin{proof}[Sketch proof of Lemma \ref{lem:fdd}.]
We follow \cite[Section 4]{Heyde11}.
The proof is via induction over $N$, and is very much inspired by the proof of \cite[Theorem 6.6.2]{MadraSlade93}, where finite-range models were considered.
The flexibility in the last argument of $n\bT$ is needed to perform the induction step.
We shall further write $nt^{\sss (j)}$ and $nT$ instead of $\lfloor nt^{\sss (j)}\rfloor$ and $\lfloor nT\rfloor$ for brevity.

To initialize the induction we consider the case $N=1$.
Since $\cnT{1}_{n\bT}(\bk_n)=\hat c_{nT}(k^{\sss(1)}_n)$, the assertion for $N=1$ is a minor generalization of Theorem \ref{thm:CritG}.
In fact, if we replace $n$ by $nT$, then instead of \eqref{eq:Dk} we have
\begin{equation}\label{eqFinDimProof0}
    nT\, \Big [1-\hat D\Big( \frac{k^{(1)}}{\sqrt{n}}\Big) \Big ]
    \rightarrow \frac{1}{2d}
    |k^{(1)}|^2
    \quad\text{as $n\to\infty$}.
\end{equation}
With an appropriate change in \eqref{eq:Proof122} we obtain \eqref{eqFinDimConv} for $N=1$ from Theorem \ref{thm:CritG}.

To advance the induction we assume that \eqref{eqFinDimConv} holds when $N$ is replaced by $N-1$.
Recalling Definition \ref{def:graphslaces}, we define for an $n$-step walk $w\in\cW_n$ and $0\le a\le b\le n$,
\begin{equation}\label{eqDefKw}
K_{[a,b]}(w):= \sum\limits_{\Gamma \in \mathcal{B}[a,b]} 
 \prod_{s, t \in [a,b]    }   \lambda U_{s,t}(w),
\end{equation}
so that $K_{[0,n]}(w) = \varphi(w)$. 
Recalling the definition of  $J_{[a,b]}(w)$ in Section \ref{sec:laceexpansion}, we obtain
\begin{equation}\label{eqDefJpi}
    \sum_{w\in\cW_n(x)} \varphi(w)J_{[0,n]}(w)=\pi_n(x),
\end{equation}
and, for any integers $0\le m\le n$ and $w\in\cW_n$,
\begin{equation}\label{eqKJKexpansion}
    K_{[0,n]}(w)=\sum_{I\ni m}K_{[0,I_1]}(w)\,J_{[I_1,I_2]}(w)\,K_{[I_2,n]}(w),
\end{equation}
where the sum is over all intervals $I=[I_1,I_2]$ of integers with either $0\le I_1<m<I_2\le n$ or $I_1=m=I_2$.
We refer to \cite[(3.13)]{SladeBook} for \eqref{eqDefJpi},
and to \cite[Lemma 5.2.5]{MadraSlade93} for \eqref{eqKJKexpansion}.
By \eqref{defCN} and \eqref{eqKJKexpansion},
\begin{equation}\label{eqFinDimProof1}
    \cnT{N}_{n\bT}(\bk_n)
    =\sum_{I\ni nt^{\sss (N-1)}}
        \sum_{w\in \cW_{nT}}
        \e^{i\bk_n\cdot\Delta w(n\bT)}
        \varphi(w)\;
        K_{[0,I_1]}(w)\,J_{[I_1,I_2]}(w)\,K_{[I_2,nT]}(w).
\end{equation}

We may restrict our attention to intervals $I$ with length $|I|=I_2-I_1\le \sqrt{n}$, as contributions from longer intervals turn out to be negligible. Further assuming that $n$ be sufficiently large so that $(nt^{\sss (N-1)}-nt^{\sss (N-2)})\vee(nt^{\sss (N)}-nt^{\sss (N-1)})\ge \sqrt{n}$, we may write the remaining contribution as 
\begin{equation}\label{eqFinDimProof2}
\begin{split}
    &\sum_{\substack{I\ni \,nt^{\sss (N-1)}\\|I|\le \sqrt{n}}}
        \cnT{N-1}_{(nt^{\sss(1)},\dots,nt^{\sss(N-2)},I_1)}
        \big(k_n^{\sss (1)},\dots,k_n^{\sss (N-1)}\big)\;
        \times\;\hat c_{nT-I_2}(k_n^{\sss (N)})\\
        &\quad{}\times\!\!\sum_{w\in \cW_{|I|}}\!
        \exp\!{\left\{ik_n^{\sss (N-1)}\cdot\, w_{nt^{\sss(N-1)}-I_1}
                +ik_n^{\sss (N)}\cdot\, (w_{I_2-I_1}-w_{nt^{\sss(N-1)}-I_1})\right\}}\;
        \varphi(w)\,
        J_{[0,|I|]}(w).
\end{split}
\end{equation}
For the two factors in the first line we can use the induction hypothesis, which yields the desired expression.
For the second line in \eqref{eqFinDimProof2}  we use $e^y=1+ O(|y|)$ and \eqref{eqDefJpi} to rewrite the second line as
\begin{equation}\label{eqFinDimProof3}
    \sum_x
        \left(1+O(|x|)\right) \pi_{|I|}(x).
\end{equation}
This latter contribution can be shown to be negligible using the strong moment bounds in Lemma 
\ref{lem:moments}. For details of the argument we refer to \cite[Thm.\ 6.6.2]{MadraSlade93} or \cite[Section 4]{Heyde11}.

\end{proof}

\begin{proof}[Sketch proof of Lemma \ref{lem:tightness}.]
We now prove tightness of the sequence $X_n$,
which is implied by the tightness criterion, 
 \cite[Proposition 5.1]{Heyde11}.
 The only thing we need to prove is 
  \cite[(5.1)]{Heyde11}.
Let us fix  $r=\frac32$.
Again we write $nt$ for $\lfloor nt \rfloor$, for brevity.
The left hand side of   \cite[(5.1)]{Heyde11}  can be written as
\begin{equation}\label{eqTightness2}
    \frac{1}{c_n n^{r} \, K^{r}}\sum_{w\in\cW_n}
    |w(nt_2)-w(nt_1)|^r\,|w(nt_3)-w(nt_2)|^r \, K_{[0,n]}(w),
\end{equation}
where $K_{[0,n]}(w)$ was defined in \eqref{eqDefKw}.
Since
\begin{equation}
    K_{[0,n]}(w)\le K_{[0,nt_1]}(w)\,K_{[nt_1,nt_2]}(w)\,K_{[nt_2,nt_3]}(w)\,K_{[nt_3,n]}(w)
\end{equation}
and, by Corollary \ref{cor:thetancn},
\begin{equation}
    c_n^{-1}\le O(1)\;c_{nt_1}^{-1}\,c_{nt_2-nt_1}^{-1}\,c_{nt_3-nt_2}^{-1}\,c_{n-nt_3}^{-1},
\end{equation}
we can bound \eqref{eqTightness2} from above by
\begin{equation}
\begin{split}
 \quad O(1)\,  
  \frac{1}{n^r} 
 \left(\xi^{(r)}(nt_2-nt_1)\right)^r\,\left(\xi^{(r)}(nt_3-nt_2)\right)^r,
\end{split}
\end{equation}
where
$$
\xi^{(r)}(n) = \Big(\frac{1}{c_n}\sum_x |x|^r c_n(x)\Big)^{1/r}.
$$
By Proposition \ref{prop-moments} we obtain,
\begin{equation}
    \left(\xi^{(r)}(nt^\ast-nt_\ast)\right)^r
    \le O(1)\,  n^{\frac{r}{2}} (t^\ast-t_\ast)^{r/2}
\end{equation}
for any $0\le t_\ast<t^\ast\le1$,
so that
\begin{equation}
    \left\langle|X_n(t_2)-X_n(t_1)|^{r}\,|X_n(t_3)-X_n(t_2)|^{r}\right\rangle_n
    \le O(1) \,(t_3-t_1)^{r}.
\end{equation}
This proves tightness of the sequence $\{X_n\}$.
\end{proof}

\appendix

\section{A random walk bound}
\begin{lemma}\label{lem:boundDind}
For $d>2$ and every  $n \geq 1$,
$$
\frac{1}{(2 \pi)^d} \int_{ [- \pi, \pi]^d} \dd k \,  \hat{D}(k)^n \, \hat{\1^0}(k)  \leq 
O(  n^{- \frac{d-1}{2}  }  ).
$$
\end{lemma} 
\begin{proof}Let $\epsilon>0$.
Using symmetry we obtain that, 
\begin{align*}
\int_{ [- \pi, \pi]^d}  \frac{\dd k}{(2 \pi)^d} \,  \hat{D}(k)^n \, \hat{\1^0}(k) &   =
d
\lim\limits_{R \rightarrow \infty} 
\frac{1}{(2 \pi)^d} \int_{ [- \pi, \pi]^d} \dd k \,   \frac{1}{d^n} \Big (   \sum\limits_{i=1}^{d} \cos( k_i)   \Big  )^n   \sqrt{ R \pi  } e^{  - k_d^2 R  } \\
&   \leq
d 
 \int_{ [- \pi, \pi]^{d-1}}\frac{ \dd k^{d-1}}{(2 \pi)^{d-1}} \,   \frac{1}{d^n} \Big (   \sum\limits_{i=1}^{d-1} \cos( k_i)  +1  \Big  )^n   \\
& =
d  \, 
\sum\limits_{k=0}^{n} \,  \binom{ n  }{ k }  \,  \frac{1}{d^{n-k}} \Big( \frac{ d-1  }{d}\Big) ^k
 \, \int_{ [- \pi, \pi]^{d-1}} \frac{ \dd k^{d-1}}{(2 \pi)^{d-1}} \,   \frac{1}{(d-1)^k} \Big (   \sum\limits_{i=1}^{d-1} \cos( k_i)  \Big  )^k \\
& \leq  C
d  \, 
\sum\limits_{k=0}^{n} \,  \binom{ n  }{ k }  \,  \frac{1}{d^{n-k}} \Big( \frac{ d-1  }{d}\Big) ^k k^{  -(d-1)/2  } \\
& 
\leq  C
d  \,  { \Big ( \frac{2}{d^{1-\epsilon}   } \Big ) }^n \, + \,  \frac{1}{(\epsilon n)^{\frac{d-1}{2}}}.
\end{align*} 
where for the second inequality we used (2.61) of \cite{Heyde11}, $C$ is some positive constant. The lemma follows by choosing $\epsilon \in (0, 1)$ appropriately. 
\end{proof}

\section{Moment estimates}\label{sec:moments} 

We prove a technical result about the control of the moments:

\begin{proposition}\label{prop-moments}
Under the assumptions of Theorem \ref{thm:BubbleCond}, there exists constants $C_1,C_2>0$ such that for all $n\in\N$ and all $r\in(0,2)$, 
\begin{equation}\label{eq:MomentBound}
	C_1\sqrt n\le \Big(\frac{1}{c_n}\sum_x |x|^r c_n(x)\Big)^{1/r}\le C_2 \sqrt n.
\end{equation}
\end{proposition}
\begin{proof}
The lower bound in \eqref{eq:MomentBound} is straightforward: we get from \eqref{eq:fcnConv} that $1-\hat c_n(k/\sqrt n)/\hat c_n(0)\to1-\exp{\big (-K|k|^2 \big )}$, and this limit is strictly positive whenever $k\neq0$. Thus for $k=(1,0,\dots,0)$, there exists $b>0$ such that for all $n$,
\[ 	b\le 1-\frac{\hat c_n(k/\sqrt n)}{\hat c_n(0)}
	=\sum_x\big[1-\cos(x_1/\sqrt n)\big]\frac {c_n(x)}{c_n}
	\le \sum_x\frac{|x_1|^r}{n^{r/2}}\frac {c_n(x)}{c_n}
	\le \frac1{n^{r/2}}\sum_x|x|^r\frac {c_n(x)}{c_n}
\]
using $1-\cos t\le |t|^r$ whenever $r\le2$. 

For the upper bound in \eqref{eq:MomentBound}, we consider the generating function 
\begin{equation}\label{eq:DefHzr}
	H_{z,r}
	:=\sum_x\sum_{n=0}^\infty|x_1|^r c_n(x)z^n
	=\sum_x\sum_{n=0}^\infty\int_0^\infty\frac{du}{u^{1+r}}\big[1-\cos(u\cdot x_1)\big]c_n(x)z^n, 
\end{equation}
determine its behaviour as $z\nearrow z_c$, and then employ a suitable Tauberian theorem. 
We split the integral of $u$ in \eqref{eq:DefHzr} in two parts, one for $u\le (z-z_c)^{1/2}$ and the other for $u> (z-z_c)^{1/2}$. 
We use that
\begin{itemize}
\item $\hat G_z(0)=O(z-z_c)^{-1}$  
\item $\hat{\Pi}_z(0)- \hat{\Pi}_z(u)=(\frac{1}{2d} + o(1))\,(1-\hat{D}(u))$,
\end{itemize}
which follows respectively from Lemma  \ref{lem:GzNull}
and Lemma \ref{lem:PiNullPiK}.
Then the arguments in \cite[Section 3]{Heyde11} give that $|H_{z,r}|\le(z_c-z)^{-1-r/2}$, and a Tauberian theorem (e.g.\ \cite[Theorem 5]{FlajoOdlyz90}) consequently gives that 
\[ \sum_x |x_1|^r c_n(x)\le n^{r/2} z_c^{-n}, \]
from which the result follows using Theorem \ref{thm:CritG}. 
\end{proof}


  \section*{Acknowledgments}
This project was supported within the FP2M federation (CNRS FR 2036), the project Labex MME-DII (ANR11-LBX-0023-01) and the \textit{Deutsche Forschungsgemeinschaft} (project numbers 444084038 and 443880457 within priority program SPP2265).


\bibliographystyle{plain}

\end{document}